\documentclass[11pt]{article}
\usepackage{amsmath}
\usepackage{amssymb, yhmath}

\usepackage{enumerate}
\usepackage{listings}

\usepackage[pdftex]{graphicx}
\usepackage{epstopdf}

\usepackage{hyperref}
\usepackage[ruled]{algorithm2e}

\DeclareMathOperator{\sgn}{sgn}
\DeclareMathOperator{\Diag}{Diag}

\DeclareMathOperator{\erf}{erf}

\usepackage{color}
\definecolor{dkgreen}{rgb}{0,0.6,0}
\definecolor{gray}{rgb}{0.5,0.5,0.5}

\usepackage{authblk}

\date{}
\newtheorem{theorem}{Theorem}[section]

\newtheorem{lemma}[theorem]{Lemma}

\newtheorem{proposition}[theorem]{Proposition}

\newenvironment{proof}[1][Proof]{\begin{trivlist}\item[\hskip \labelsep {\bfseries #1.}]}{$\Box$\end{trivlist}}

\newcommand{\diag}{\operatorname{diag}}

\newcommand{\real}{\mathbb R}
\newcommand{\dd}{\mathrm{d}}

\numberwithin{equation}{section}

\title{Conjugate gradient based acceleration for inverse problems}

\author[1,3]{Sergey Voronin}
\author[2]{Christophe Zaroli}
\author[1]{Naresh P. Cuntoor}
\affil[1]{Intelligent Automation Inc, Rockville, MD, USA}
\affil[2]{Institut de Physique du Globe de Strasbourg, UMR 7516, Universit\'e de Strasbourg, EOST/CNRS, France}
\affil[3]{Tufts University Department of Mathematics, Medford, MA, USA}

\date{\today}

\begin{document}

\maketitle

\begin{abstract}
The conjugate gradient method is a widely used algorithm for the numerical solution of a system of linear equations. 
It is particularly attractive because it allows one to take 
advantage of sparse matrices and produces (in case of infinite precision arithmetic) 
the exact solution after a finite number of iterations. It is thus well suited 
for many types of inverse problems. On the other hand, the method requires the computation of the gradient. Here difficulty 
can arise, since the functional of interest to the given inverse problem 
may not be differentiable. In this paper, we review two approaches 
to deal with this situation: iteratively reweighted least squares and convolution smoothing. We apply 
the methods to a more generalized, two parameter penalty functional. 
We show advantages of the proposed algorithms using examples from  
a geotomographical application and for synthetically constructed multi-scale reconstruction and regularization parameter estimation.   
\end{abstract}

\section{Introduction}

Consider the linear system $Ax=b$, where
$A\in\mathbb{R}^{m\times n}$ and $b\in\mathbb{R}^m$. Often, in linear 
systems arising from physical inverse problems, we have 
more unknowns than data: $m \ll n$ \cite{tarantola82} and the right hand side of the 
system corresponding to the observations or measurements is noisy. In such a setting, 
it is common to use regularization by introducing a constraint on the solution, 
both to account for the 
possible ill-conditioning of $A$ and noise in $b$ and for the 
lack of data with respect to the number of unknown variables in the linear system.
A commonly used constraint is imposed via a penalty on the norm of the solution, 
either forcing the sum of certain powers of the absolute values of coefficients 
to be bounded or 
for many of the coefficients to be zero, i.e. sparsity: to require the 
solution $x$ to have few nonzero elements compared to the dimension of $x$.
To account for different kinds of scenarios, we consider here the generalized functional:
\begin{equation}
\label{eq:lp_funct}
   F_{l,p}(x) = \|Ax - b\|_l^l
    + \lambda \|x\|_p^p = 
    \displaystyle\sum_{i=1}^m \left| \displaystyle\sum_{j=1}^n A_{ij} x_j - b_i \right|^l
    + \lambda \sum_{k=1}^n |x_k|^p,
\end{equation}
for $1 \leq l,p \leq 2$. For example, when $l=p=2$, the familiar Tikhonov regularization is recovered. 
When $l=2$ and $p=1$, we obtain the least squares problem 
with the convex $\ell_1$ regularizer, governed by the regularization parameter $\lambda>0$ 
\cite{ingrid_thresholding1}, commonly used for sparse signal recovery. 
In addition, \eqref{eq:lp_funct} allows us to impose a non-standard penalty on the 
residual vector, $r = Ax - b$. This is useful e.g. in cases, where we want to impose a higher 
penalty on any outliers which may be present in $b$. 
(Since the $l=2$ case is commonly utilized, we denote $\tilde{F}_p(x) := F_{2,p}(x)$). 
For any $p\geq1$, the map 
$\|x\|_p := \left( \sum_{k=1}^n |x_k|^p \right)^{\frac{1}{p}}$
(for any $x\in\mathbb{R}^n$) is called the $\ell_p$-norm on 
$\mathbb{R}^n$. For $p=1$, the $\|\cdot\|_1$ norm is called an $\ell_1$ norm and is convex.
As $p \to 0$, the right term of this functional approximates 
the count of nonzeros or the so-called $\ell_0$ ``norm'': 
\begin{equation*}
   \|x\|_0 
   = \lim_{p \to 0} \|x\|_p 
   = \lim_{p \to 0} 
        \left( \sum_{k=1}^n |x_k|^p \right)^{1/p}.
\end{equation*}
For different $p$ values this measure is plotted in Figure \ref{fig:fvals_to_diff_p}. 

\begin{figure*}[h!]
\centerline{
\includegraphics[scale=0.3]{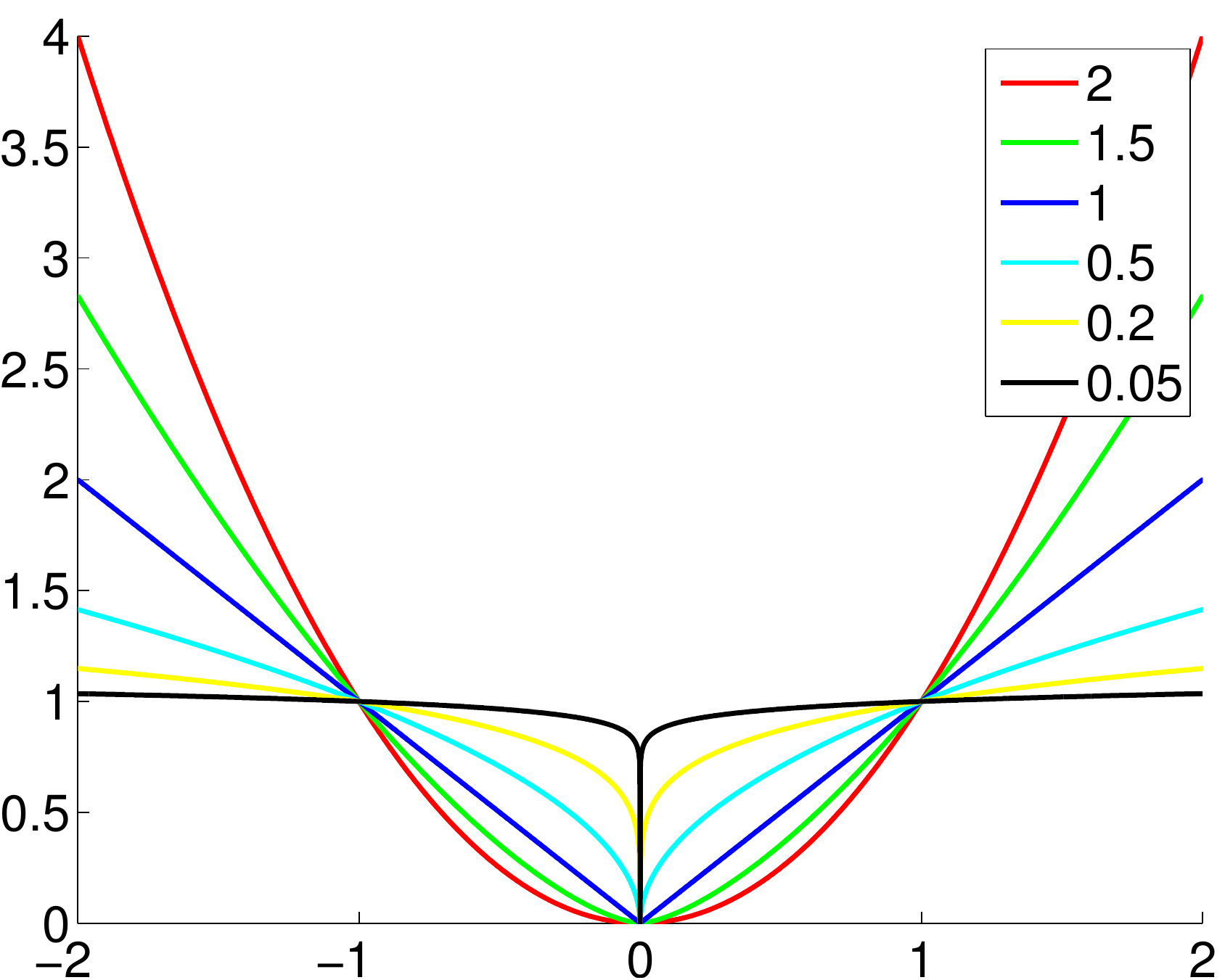}
}
\caption{$|x|^p$ plotted for different values of $p$; as $p \to 0$, the plot approaches an
indicator function.}
\label{fig:fvals_to_diff_p}
\end{figure*}
  
The non-smoothness of the family of functionals $F_{l,p}(x)$ complicates their minimization 
from an algorithmic point of view. The non-smooth part of \eqref{eq:lp_funct}
is due to the absolute value function $g(x_k) = |x_k|$ or of $h(r_k) = |(A x - b)_k|$, 
or both, depending on the values of $l$ and $p$. 
Because the gradient of $F_{l,p}(x)$ cannot be obtained when $l$ or $p$ are less than $2$, 
different minimization techniques such as sub-gradient methods are 
frequently used \cite{ShorMinimization}.
For the convex $l=2, p=1$ case in \eqref{eq:lp_funct}, various thresholding based 
methods have become popular. A particularly successful example 
is the soft thresholding 
based method FISTA \cite{MR2486527}. This algorithm is an accelerated version of a soft thresholded 
Landweber iteration \cite{MR0043348}:
\begin{equation}
\label{eq:ista}
x^{n+1} = \mathbb{S}_{\frac{\lambda}{2}}(x^n + A^T b - A^T A x^n) .
\end{equation}
The soft thresholding function $\mathbb{S}_{\lambda}:\mathbb{R}^n\to \mathbb{R}^n$ \cite{ingrid_thresholding1} is defined by
\begin{equation*}
   \left(\mathbb{S}_{\lambda}(x)\right)_k = \sgn(x_k) \max{\{0, |x_k| - \lambda\}}, \ \forall\, k=1,\ldots,n,\ \forall\, x\in\mathbb{R}^n.
\end{equation*}
The scheme \eqref{eq:ista} is known to converge from some initial guess, but slowly, to the 
$\ell_1$ minimizer \cite{ingrid_thresholding1}.
The thresholding in \eqref{eq:ista} is performed on 
$x^n - \nabla_x (\frac{1}{2} \|Ax^n - b\|_2^2) = x^n - A^T (A x^n - b)$, which is a very 
simple gradient based scheme with a constant line search \cite{EnglRegularization}. 
The thresholding based schemes typically 
require many iterations to converge and this is costly 
(due to the many matrix-vector mults required) when $A$ is large. Moreover, the regularization 
parameter $\lambda$ is often not known in advance. Instead, it is frequently estimated using 
a variant of the L-curve technique \cite{hansen1999curve} together with a 
continuation scheme, where $\lambda$ is iteratively decreased, reusing the previous solution 
as the initial guess at the next lower $\lambda$. This requires many iterations.  

In this article, we discuss two approaches to obtaining approximate solutions to  
\eqref{eq:lp_funct} using an accelerated conjugate gradient approach, 
with a specific focus on the case of large matrix $A$, where the 
use of thresholding based techniques is expensive, due to the many iterations required. 
In contrast, our methods accomplish similar work in fewer iterations, because each iteration 
is more powerful than that of a thresholding scheme. 
We consider specifically the case $1 \leq l,p \leq 2$ since for $0<l,p<1$, 
\eqref{eq:lp_funct} is not convex. However, the minimization of non-smooth non-convex functions 
has been shown to produce good results in some compressive sensing applications 
\cite{Chartrand09fastalgorithms} and our methods can be applied also to the non-convex case, 
as long as care is taken to avoid local minima.

The first approach is based on the conjugate gradient acceleration of the 
reweighted least squares idea developed in \cite{voronin2012regularization} with 
further developments and analysis given in \cite{fornasier2016conjugate}. 
The approach is based on a two norm approximation of the absolute value function:
\begin{equation*}
|x_k| = \frac{x_k^2}{|x_k|} = \frac{x_k^2}{\sqrt{x_k^2}} \approx \frac{x_k^2}{\sqrt{x_k^2 + \epsilon^2}}
\end{equation*}
where in the rightmost term, a small $\epsilon \neq 0$ is used, 
to insure the denominator is finite, regardless of the value of $x_k$. 
Thus, at the $n$-th iteration, a reweighted $\ell_2$-approximation to the 
$\ell_1$-norm of $x$ is of the form:
\begin{equation*}
\|x\|_1 \approx \displaystyle\sum_{k=1}^N \frac{x_k^2}{\sqrt{(x^n_k)^2 + \epsilon_n^2}} = \displaystyle\sum_{k=1}^N \tilde{w}^n_k x_k^2
\end{equation*}
where the right hand side is a reweighted two-norm with weights:
\begin{equation*}
\tilde{w}^n_k = \frac{1}{\sqrt{(x^n_k)^2 + \epsilon_n^2}}.
\end{equation*}
It follows that $\sum_k \tilde{w}^n_k (x^n_k)^2$ is a close approximation
to $\|x^n\|_1$, which proceeds to $\|x\|_1$ as $n \to \infty$. 
Given the smooth approximation resulting 
from the reweighted two norm, the gradient acceleration idea is then built on top 
of the least squares approximations. A slight generalization of the weights makes the approach 
applicable to \eqref{eq:lp_funct} with $l=2$. With the aid of results from \cite{scales1988fast} and the assumption that the residuals 
are nonzero, we are able to extend the algorithm to the general case in \eqref{eq:lp_funct}.

The second approach is based on smooth approximations to the non-smooth 
absolute value function $g(t) = |t|$, computed via convolution with a Gaussian function, 
described in detail in \cite{voronin2014convolution}.
Starting from the case $l=2$, we replace the non-smooth objective function 
$\tilde{F}_{p}(x)$ by a smooth functional $H_{p,\sigma}(x)$, which is close to $\tilde{F}_{p}(x)$ in value 
(as the parameter $\sigma \to 0$). Since the approximating functional 
$H_{p,\sigma}(x)$ is smooth, 
we can compute its gradient vector $\nabla_x H_{p,\sigma}(x)$ and Hessian matrix 
$\nabla^2_x H_{p,\sigma}(x)$. We are then able to use gradient based 
algorithms such as conjugate gradients to approximately minimize $\tilde{F}_{p}(x)$ 
by working with the approximate functional and gradient pair. 
We also apply the approach to the general case 
(where we may have $l \neq 2$) in \eqref{eq:lp_funct}, with the assumption that the residuals are nonzero.

In this paper, we describe the use of both acceleration approaches and the generalized functional, and give 
some practical examples from Geophysics and of wavelet based model reconstructions.

\vspace{2.mm}

\section{Iteratively Reweighted Least Squares}
The iteratively reweighted least squares (IRLS) method, was originally presented 
in \cite{daubechies2010iteratively}. The algorithm presented here is from the work in 
\cite{voronin2012regularization}.
Several new developments have recently emerged. In particular, \cite{fornasier2016conjugate} provides 
the derivations for the practical implementation of the IRLS CG scheme 
(without running the CG algorithm to convergence at each iteration) 
while \cite{behtash2013convergence} provides some stability and convergence arguments in the presence of noise. In this 
article, we survey the method and present the extension of the algorithm to \eqref{eq:lp_funct}, 
without going into the mathematical details for the convergence arguments, 
which are provided in the above references and can be extended to \eqref{eq:lp_funct} with our assumptions.     
The basic idea of IRLS consists of a series of smooth approximations to the absolute value function:
\begin{equation*}
\|x\|_1 \approx \displaystyle\sum_{k=1}^N \frac{x_k^2}{\sqrt{(x^n_k)^2 + \epsilon_n^2}} = \displaystyle\sum_{k=1}^N \tilde{w}^n_k x_k^2
\end{equation*}
where the right hand side is a reweighted two-norm with weights:
\begin{equation}
\label{eq:irls_scheme_weights_ell1}
\tilde{w}^n_k = \frac{1}{\sqrt{(x^n_k)^2 + \epsilon_n^2}}.
\end{equation}
In \cite{voronin2015iteratively}, the non-CG version of the IRLS algorithm is considered with the generalized weights:
\begin{equation}
\label{eq:irls_scheme_weights}
w^n_k = \frac{1}{\left[(x^n_k)^2 + \epsilon_n^2\right]^\frac{2-p}{2}}.
\end{equation}
This results in the iterative scheme:
\begin{equation}
\label{eq:irls_scheme}
x^{n+1}_k = \frac{1}{1 + \frac{1}{2} \lambda p w^n_k} \left(x^n_k + (A^T b)_k - (A^T A x^n)_k\right) \quad \mbox{for} \quad k=1,\dots,N,
\end{equation}
which converges to the minimizer of \eqref{eq:lp_funct} for $l=2$ and $1 \leq p \leq 2$.
The iteratively reweighted least squares (IRLS) algorithm given by scheme 
\eqref{eq:irls_scheme} with weights \eqref{eq:irls_scheme_weights} follows from 
the construction of a surrogate functional \eqref{eq:Gdef}, 
as per Lemma \ref{lem:irls_surrogate_functional} below.

\begin{lemma}
\label{lem:irls_surrogate_functional}
Define the surrogate functional:
\begin{eqnarray}
\label{eq:Gfunct1}
\quad
G(x,a,w,\epsilon) &=& \|Ax - b\|_2^2 - \|A(x - a)\|_2^2 + \|x - a\|_2^2 \\
\label{eq:Gdef}
&+& \displaystyle\sum \frac{1}{2} \lambda \left( p w_k \left((x_k)^2 + \epsilon^2\right) + (2 - p) (w_k)^{\frac{p}{p - 2}} \right) \\
\nonumber
\end{eqnarray}
where $\epsilon_n = \min\left( \epsilon_{n-1}, (\|x^n - x^{n-1}\|_2 + \alpha)^{\frac{1}{2}} \right)$ with  
$\alpha \in (0,1)$. 
Then the minimization procedure
$~w^{n} = \arg\min_w G(x^{n},a,w,\epsilon_{n})~$
defines the iteration dependent weights:
\begin{equation}
\label{eq:witer}
w^{n}_k = \frac{1}{\left[(x^{n}_k)^2 + \epsilon_{n}^2\right]^\frac{2-p}{2}}.
\end{equation}
In addition, the minimization procedure
$x^{n+1} = \arg\min_x G(x,x^n,w^n,\epsilon_n)$, produces the iterative scheme:
\begin{equation}
\label{eq:xiter}
x^{n+1}_k = \frac{1}{1 + \frac{1}{2} \lambda p w^n_k} \left( (x^n)_k - (A^T A x^n)_k +  (A^T b)_k \right).
\end{equation}
which converges to the minimizer of \eqref{eq:lp_funct} for $l=2$ and $1 \leq p \leq 2$.
\end{lemma}
The proof of the lemma is given in \cite{voronin2015iteratively}. The construction of 
the non-increasing sequence $(\epsilon_n)$ is important for convergence analysis. On the other hand, different choices 
can be used for implementations. The auxiliary $G(\dots)$ function is also used for the convergence proof 
and for the derivation of the algorithm. In particular, it is chosen to yield $\|x^n - x^{n-1}\| \to 0$ and 
to conclude the boundedness of the sequence of iterates $(x^n)$.
In the same paper, the FISTA style acceleration of the scheme is shown. 
In practice, however, a large number of iterations may be required for convergence and so the CG acceleration of 
the above scheme is of particular interest. 

\subsection{Conjugate gradient acceleration}

Acceleration via CG is accomplished by modifying the auxiliary functional \eqref{eq:Gfunct1}.
If we instead set, 
\begin{equation*}
G(x,w,\epsilon) = \|Ax - b\|_2^2 + \lambda \displaystyle\sum_{k=1}^N \left[ p w_k \left(x_k^2 + \epsilon^2\right) + (2-p) w_k^{\frac{p}{p-2}} \right],
\end{equation*}
then the two minimization problems:
\begin{equation*}
w^{n+1} = \arg\min_w G(x^{n+1},w,\epsilon_{n+1}) \quad ; \quad x^{n+1} = \arg\min_x G(x,w^n,\epsilon_n) 
\end{equation*}
give the same iteration dependent weights in \eqref{eq:witer} and the iterative scheme:
\begin{equation}
\label{eq:irls_cg1}
x^{n+1} = \arg\min_x \left\{ \|Ax - b\|_2^2 + \frac{1}{2} \lambda p \displaystyle\sum_{k=1}^N w^n_k x_k^2 \right\}.
\end{equation}
The details of convergence are given in \cite{voronin2012regularization}.
The choice of the non-increasing $(\epsilon_n)$ sequence is again crucial for convergence analysis, although simpler 
choices (e.g. $\epsilon_n = \|x^n - x^{n-1}\|$ can be programmed in practice). 
The choice 
\begin{equation}
\label{eq:irls_sys_epsilon_update1}
\epsilon^n = \min\left(\epsilon^{n-1}, |G(x^{n-2},w^{n-2},\epsilon_{n-2}) - G(x^{n-1},w^{n-1},\epsilon_{n-1})|^{\frac{\gamma}{2}} + \alpha^n \right),
\end{equation}
with $\alpha \in (0,1)$ and $0 < \gamma < \frac{2}{4 - p^2}$ is taken for showing convergence to the 
minimizer in \cite{voronin2012regularization}. 

We now look more closely at the optimization problem in \eqref{eq:irls_cg1} above. In particular, notice that we can 
write:
\begin{equation*}
\displaystyle\sum_{k=1}^N \frac{1}{2} \lambda p w^n_k x_k^2 = \displaystyle\sum_{k=1}^N (D^n_{kk})^2 x_k^2 = \|D^n x\|_2^2 
\end{equation*} 
where $D^n$ is an iteration dependent diagonal matrix with elements 
$D^n_{kk} = \sqrt{\frac{1}{2} \lambda p w^n_k}$.
This observation allows us to write the solution to the optimization problem at each iteration in terms of a linear system:
\begin{equation}
\label{eq:irls_cg2}
x^{n+1} = \arg\min_x \left\{ \|Ax - b\|_2^2 + \|D^n x\|_2^2 \right\} \implies \left(A^T A + (D^n)^T (D^n)\right) x^{n+1} = A^T b.
\end{equation}
In turn, the system in \eqref{eq:irls_cg2} can be solved using CG iterations. In practice, we need only an 
approximate solution to the above linear system at each iteration so the amount of CG iterations can be less than $5$ 
at each iteration $n$. In \cite{fornasier2016conjugate}, it is shown that this procedure (involving inexact CG 
solutions) gives convergence to the minimizer. 
Notice also that it is not necessary to build up the $D^n$ matrix. Instead, the matrix is applied to 
vectors at each iteration, which can be done using an array computation. 
In \cite{voronin2012regularization}, the application of the Woodbury matrix identity is explored, which 
can aid in cases where IRLS is applied to short and wide or tall and thin matrices.

\subsection{Application to generalized residual penalty}

Next, we consider the application of the IRLS scheme to \eqref{eq:lp_funct}.
In \cite{scales1988fast} an IRLS algorithm for the minimization of 
$\min_x \displaystyle\sum_{i=1}^m \left| \displaystyle\sum_{j=1}^n A_{ij} x_j - b_i \right|^l$ was developed. 
The generalized IRLS equations resulting from this problem are $A^T R^n A x^{n+1} = A^T R^n b$, where 
$R^n$ is a diagonal matrix with elements $l |r^n_i|^{l-2}$, where $r^n_i = (Ax^n - b)_i = \displaystyle\sum_{j=1}^n A_{ij} x_j - b_i$. 
This system was derived by setting the gradient to zero where it was assumed that $r_i \neq 0$ for all $i$. In this case 
\cite{scales1988fast}:
\begin{equation}
\label{eq:scales_derivative_calc1}
\frac{\partial}{\partial x_k} \displaystyle\sum_{i=1}^m |r_i(x)|^l = \displaystyle\sum_{i=1}^m l \sgn(r_i) |r_i|^{l-1} A_{ik} = 
\displaystyle\sum_{i=1}^m r_i l |r_i|^{l-2} A_{ik} = [A^T R (Ax - b)]_k
\end{equation}
where we make use of $\sgn(r_i) = \frac{r_i}{|r_i|}$. When $l=2$, the familiar normal equations are recovered. 
Notice that $r_i \neq 0$ for all $i$ is a strong but plausible assumption to make 
(for example, to enforce this, we can add a very small amount of additive noise to the right hand side; the effect of 
noise on stability and convergence of IRLS has been analyzed in \cite{behtash2013convergence}).  
However, the same cannot be said of $x_i$ e.g. for sparse solutions, where many components can equal zero.
If, with the assumption $r_i \neq 0$ for all $i$, we use the surrogate functional, 
\begin{equation*}
\tilde{G}(x,w,\epsilon) = \displaystyle\sum_{i=1}^m \left| \displaystyle\sum_{j=1}^n A_{ij} x_j - b_i \right|^l + \lambda \displaystyle\sum_{k=1}^N \left[ p w_k \left(x_k^2 + \epsilon^2\right) + (2-p) w_k^{\frac{p}{p-2}} \right]
\end{equation*}
it then follows that the resulting algorithm for the minimization of \eqref{eq:lp_funct} can be written as:
\begin{equation}
\label{eq:irls2}
\left(A^T R^n A + (D^n)^T (D^n)\right) x^{n+1} = A^T R^n b
\end{equation}
with the same diagonal matrix $D^n$ as before and with the diagonal matrix $R^n$ with 
diagonal elements $|r^n_i|^{l-2}$ (for $i$ where $|r^n_i| < \epsilon$, we can set the entry to $\epsilon^{l-2}$ with the 
choice of $\epsilon$ user controllable, tuned for a given application). 
As before, at each iteration $n$, the system in \eqref{eq:irls2} can be solved 
(approximately) using a few iterations of the CG algorithm (or some variant of the method, such as LSQR \cite{paige1982lsqr}). 
Below, we present the basic version of the IRLS CG algorithm. In practice, the parameter sequence $\epsilon_n$ for $D$ can be set to e.g. 
$\|x^n - x^{n-1}\|$ or the update in \eqref{eq:irls_sys_epsilon_update1} can be used. 
Fixing $\epsilon$ for $R$ is common. Varying it may be useful in the case that $p$ or $l$ less than $1$ are  
chosen, resulting in a non-convex problem. For large problems it is important not to form $D$ and $R$ matrices 
explicitly. Instead vectors $\vec{d}$ and $\vec{r}$ can be used to hold their diagonal elements (i.e. 
$D^{n} x = \vec{d}^{n} \circ x$).

\begin{algorithm}[!ht]
\SetKwInOut{Input}{Input}
\SetKwInOut{Output}{Output}
\caption{IRLS CG Algorithm
\label{algo:irlscg}}
\Input{An $m\times n$ matrix $A$, 
an initial guess $n \times 1$ vector $x^0$, 
a parameter $\lambda < \|A^T b\|_\infty$, 
a parameter $l\in[1,2]$, 
a parameter $p\in[1,2]$, 
a parameter $\epsilon\in(0,1)$,
a maximum number of iterations to perform $N$ and a maximum number 
of local CG iterations to perform $N_l$.}
\Output{A vector $\bar{x}$, 
close to either the global or local minimum of $F_{l,p}(x)$, 
depending on choice of $p,l$.}
\BlankLine
Set $\epsilon_{1} = \epsilon$. \\ 
\For{$n=1,\dots N$}{
Set weights $w^n_k = \frac{1}{\left[(x^{n}_k)^2 + \epsilon_n^2\right]^\frac{2-p}{2}}$. \\
Initialize vectors $\vec{d}^n$ and $\vec{r}^n$ for the diagonal elements of diagonal matrices $D^n$ and $R^n$, 
with $\vec{d}^n_k = \sqrt{\frac{1}{2}\lambda p w^n_k}$ and 
$\vec{r}^n_k = l |r_k|^{l-2}$, with $r_k = (A x^n - b)_k$, for $|r_k| > \epsilon$; and $r^n_k = l |\epsilon|^{l-2}$ for $|r_k| < \epsilon$. \\
Run $N_l$ iterations of CG on the system $\left(A^T R^n A + (D^n)^T (D^n)\right) x^{n+1} = A^T R^n b$. \\
Compute $\epsilon_{n+1}$. 
}
\end{algorithm}

\vspace{2.mm}

\section{Approximate mollifier approach via convolution}

In mathematical analysis, a smooth function $\psi_\sigma:\real\to\real$ 
is said to be a (non-negative) \textit{mollifier} 
if it has finite support,
is non-negative $(\psi \geq 0)$, 
and has area $\int_{\mathbb{R}} \psi(t) \dd t = 1$ \cite{DenkowskiNonlinearAnalysis}. 
For any mollifier $\psi$ and any $\sigma>0$, define the parametric function $\psi_\sigma:\mathbb{R}\to\mathbb{R}$ by: 
$\psi_{\sigma}(t) := \frac{1}{\sigma}\psi\left(\frac{t}{\sigma}\right)$, for all $t\in\real$.
Then $\{\psi_{\sigma}:\sigma>0\}$ is a family of mollifiers, 
whose support decreases 
as $\sigma \to 0$, but the volume under the graph always remains equal to one.
We then have the following important lemma for the approximation of 
functions, whose proof is given in \cite{DenkowskiNonlinearAnalysis}.
\begin{lemma}
For any continuous function $g\in L^1(\Theta)$ 
with compact support and $\Theta\subseteq\mathbb{R}$,
and any mollifier $\psi:\mathbb{R}\to\mathbb{R}$,
the convolution $\psi_{\sigma} * g$, which is the function defined
by:
\begin{equation*}
   (\psi_{\sigma} * g) (t)
   := \int_{\mathbb{R}} \psi_{\sigma} (t-s) g(s) \mathrm{d}s
   = \int_{\mathbb{R}} \psi_{\sigma} (s) g(t-s) \mathrm{d}s,
   \ \forall\, t\in\mathbb{R},
\end{equation*} 
converges uniformly to $g$ on $\Theta$, as $\sigma\to 0$.
\end{lemma}

\subsection{Smooth approximation to the absolute value function}
Motivated by the above results, we will use convolution with 
approximate mollifiers to approximate the absolute value function 
$g(t) = |t|$ (which is not in $L^1(\mathbb{R})$) with a smooth function.
We start with the Gaussian function $K(t) = \frac{1}{\sqrt{2 \pi}} \exp\left(-\frac{t^2}{2}\right)$ (for all $t\in\real$),
and introduce the $\sigma$-dependent family: 
\begin{equation}
\label{eq:gaussian_mollifier1}
   K_{\sigma}(t) := 
   \frac{1}{\sigma} K\left(\frac{t}{\sigma}\right) = 
   \frac{1}{\sqrt{2 \pi \sigma^2}} 
      \exp\left( -\frac{t^2}{2\sigma^2} \right), 
   \quad\forall\, t\in\mathbb{R}.
\end{equation} 
This function is an approximate mollifier since strictly speaking, it does not have finite 
support. However, this function is coercive, that is, 
for any $\sigma>0$, $K_{\sigma}(t) \to 0$ as $|t|\to\infty$. 
In addition, we have that 
$\int_{-\infty}^{\infty} K_{\sigma}(t) \, \mathrm{d}t=1$ for all $\sigma>0$:
\begin{eqnarray*}
   \int_{-\infty}^{\infty} \! K_{\sigma}(t) \, \mathrm{d}t 
   &=& 
   \frac{1}{\sqrt{2 \pi \sigma^2}} 
     \int_{\mathbb{R}} \exp\left(-\frac{t^2}{2\sigma^2} \right) \, \mathrm{d}t 
    = 
   \frac{2}{\sqrt{2 \pi \sigma^2}} 
     \int_{0}^{\infty} \exp\left(-\frac{t^2}{2\sigma^2} \right) \, \mathrm{d}t 
\\ 
   &=& \frac{2}{\sqrt{2 \pi \sigma^2}} 
     \int_{0}^\infty \exp(-u^2) \sqrt{2} \sigma \, \mathrm{d}u
    = \frac{2}{\sqrt{2 \pi \sigma^2}} \sqrt{2} 
      \sigma \frac{\sqrt{\pi}}{2}
    = 1.
\end{eqnarray*}
Figure \ref{fig:Ksigma_plot} below presents a plot of the function 
$K_{\sigma}$ in relation to the particular choice $\sigma = 0.01$. 
We see that $K_{\sigma}(t) \geq 0$ and $K_{\sigma}(t)$ 
is very close to zero for $|t| > 4 \sigma$. 
In this sense, the function $K_\sigma$ is an approximate mollifier. 

\begin{figure*}[!ht]
\centerline{
\includegraphics[scale=0.25]{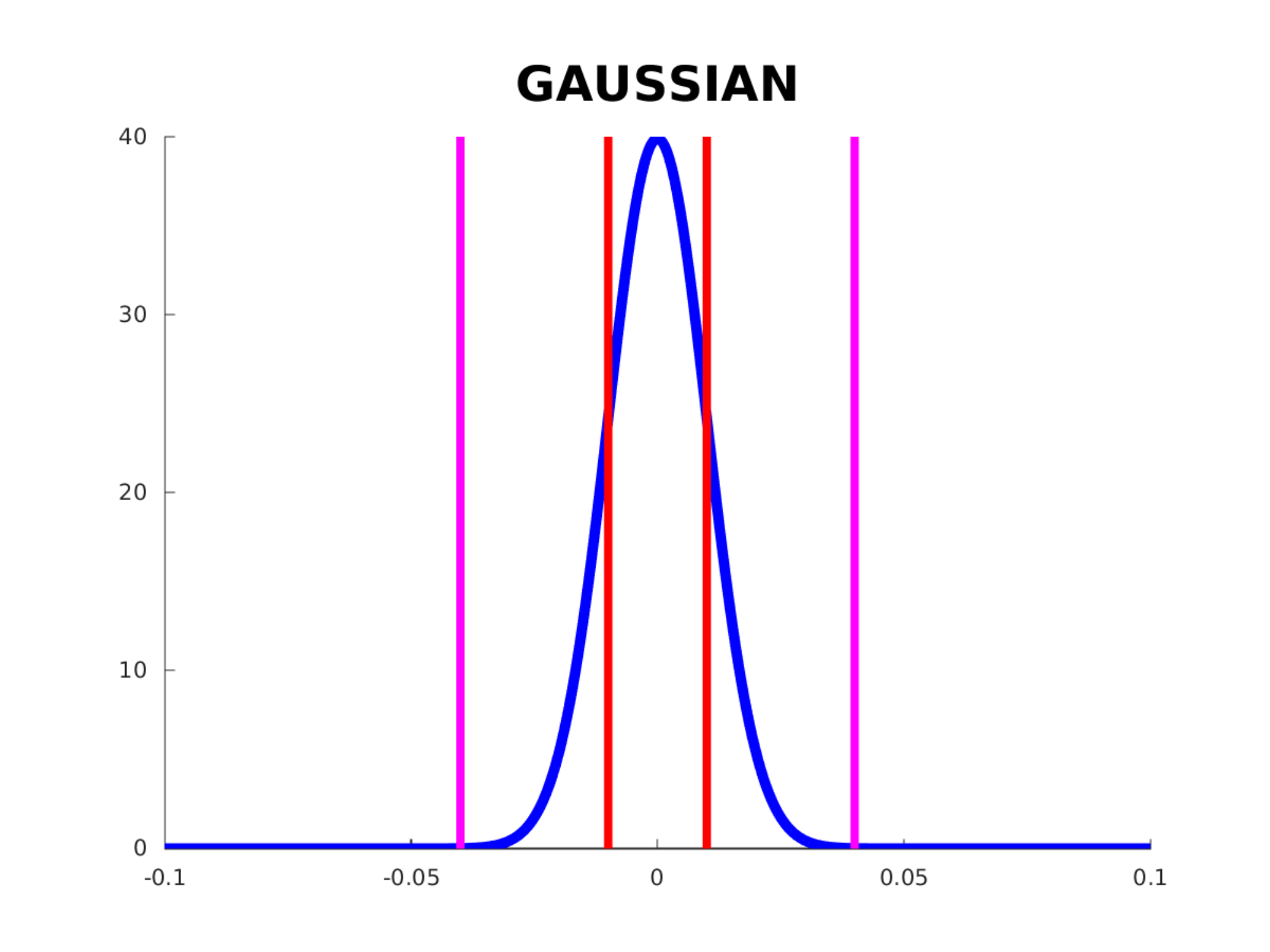}
\includegraphics[scale=0.25]{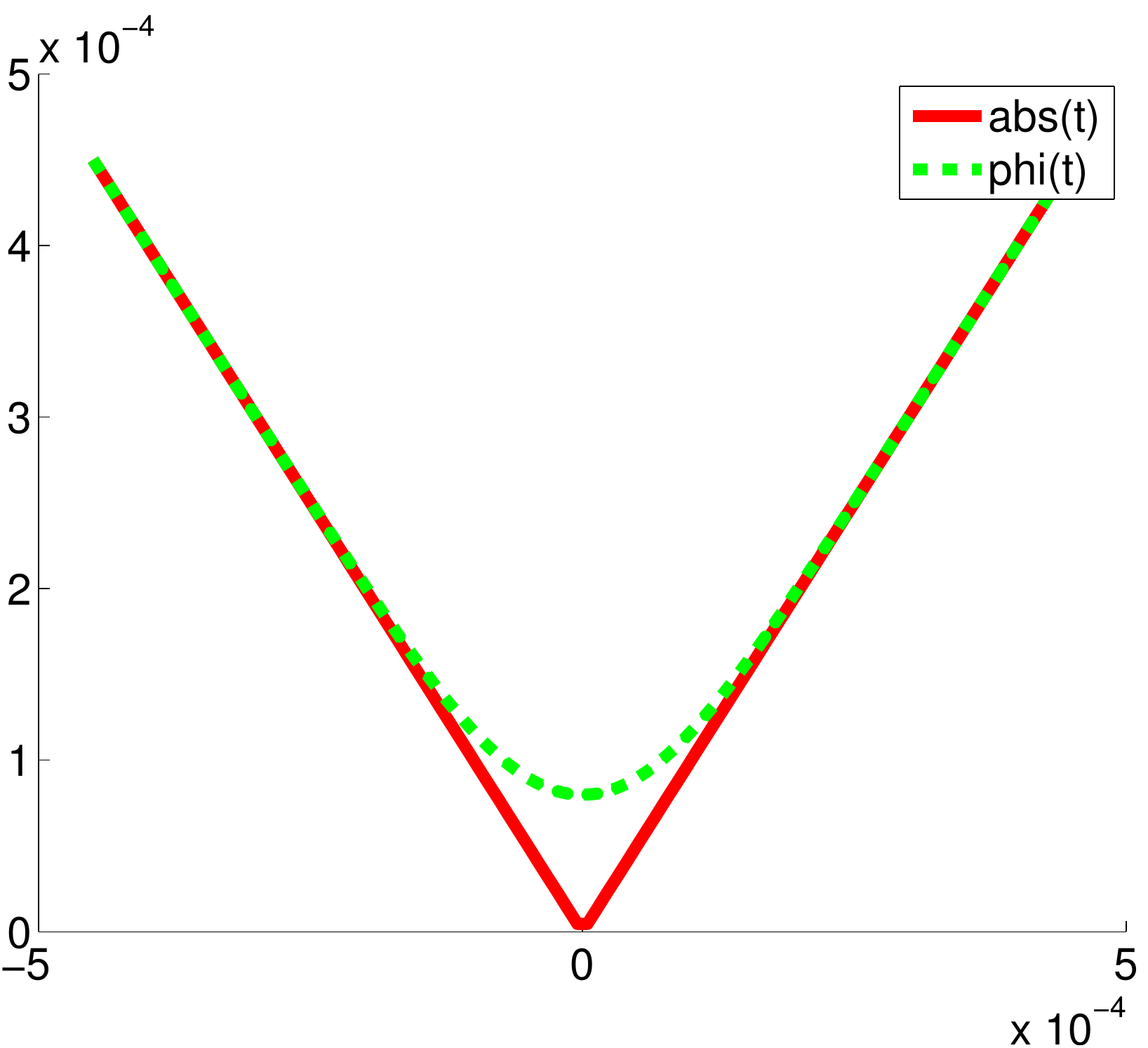}
\includegraphics[scale=0.25]{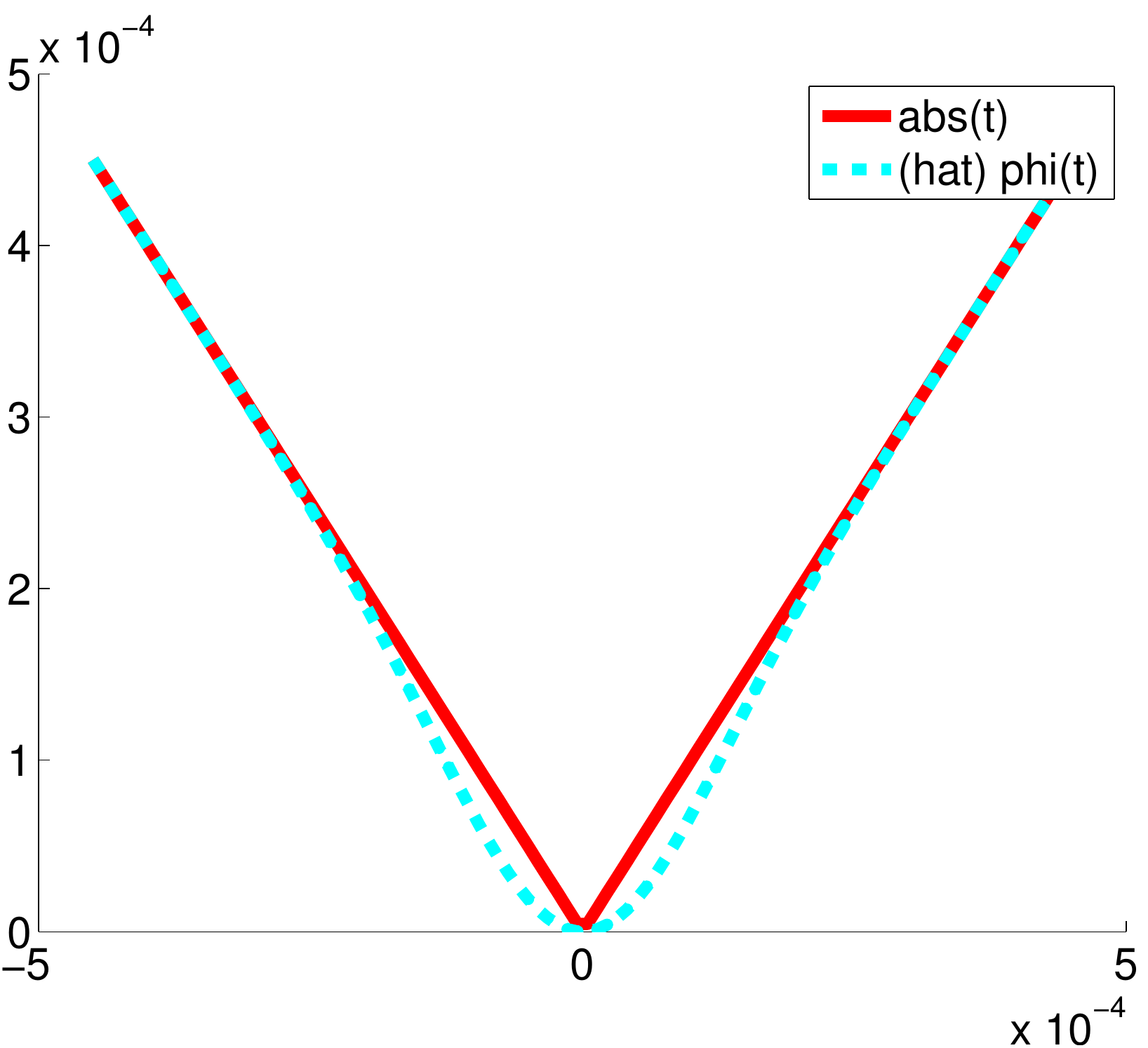}
}
\caption{$K_{\sigma}(t)$ and vertical lines at $(-\sigma, \sigma)$ and 
$(-4 \sigma, 4 \sigma)$ for $\sigma = 0.01$. Convolution based approximations to $|t|$.\label{fig:Ksigma_plot}}
\end{figure*}

Let us now compute the limit $\lim_{\sigma \to 0} K_{\sigma}(t)$. 
For $t = 0$, it is immediate that $\lim_{\sigma \to 0} K_{\sigma}(0) = \infty$. For $t\neq0$, we use 
l'H\^{o}pital's rule:
\begin{equation*}
   \lim_{\sigma \to 0} K_{\sigma}(t) = 
   \lim_{\sigma \to 0} \frac{1}{\sqrt{2 \pi \sigma^2}} 
     \exp\left( -\frac{t^2}{2\sigma^2} \right) = 
   \lim_{\gamma \to \infty} \frac{\gamma}{\sqrt{2 \pi} 
     \exp\left(\frac{\gamma^2t^2}{2}\right)} = 
   \frac{1}{\sqrt{2 \pi}} \lim_{\gamma \to \infty} 
     \frac{1}{\gamma\, t^2 \exp\left(\frac{\gamma^2t^2}{2}\right)} = 0,
\end{equation*}
with $\gamma = \frac{1}{\sigma}$. 
We see that $K_{\sigma}(t)$ behaves 
like a Dirac delta function $\delta_0(x)$ with unit integral over 
$\mathbb{R}$ and the same pointwise limit. 
Thus, for small $\sigma>0$, 
we expect that the absolute value function can be approximated 
by its convolution with $K_\sigma$, i.e., 
\begin{equation}
\label{eq:x_k_approx1}
   |t| \approx 
   \phi_\sigma(t),
   \quad\forall\, t\in\real,
\end{equation}
where the function $\phi_\sigma:\real\to\real$ is defined as the
convolution of $K_\sigma$ with the absolute value function:
\begin{equation}
\label{eq:phisigma}
   \phi_\sigma(t):=
   (K_{\sigma} * |\cdot|)(t) =
      \frac{1}{\sqrt{2 \pi \sigma^2}} 
      \int_{-\infty}^{\infty} |t - s|
      \exp\left( -\frac{s^2}{2\sigma^2} \right) 
      \mathrm{d}s,
   \quad\forall\, t\in\real.
\end{equation}
\noindent
We show in Proposition \ref{prop:conv} below, 
that the approximation in \eqref{eq:x_k_approx1} converges 
in the $L^1$ norm (as $\sigma \to 0$).
The advantage of using this approximation is that 
$\phi_\sigma$, unlike the absolute value function, 
is a smooth function. 

Before we state the convergence result in 
Proposition \ref{prop:conv}, 
we express the convolution integral and 
its derivative in terms of the well-known error function \cite{HandbookOfMathFunctions}.

\begin{lemma}
\label{lem:convolution}
For any $\sigma>0$, define $\phi_\sigma:\real\to\real$ 
as in \eqref{eq:phisigma}
Then we have that for all $t\in\real$:
\begin{align}
\label{eq:convolution}
   \phi_\sigma(t) 
   \,=&\ t \erf\left( \frac{t}{\sqrt{2} \sigma} \right) 
     + \sqrt{\frac{2}{\pi}} \sigma 
         \exp\left(-\frac{t^2}{2 \sigma^2}\right),
\\
\label{eq:phisigma_derivative}
   \frac{\dd}{\dd t}\,\phi_\sigma(t) 
   \,=&\ 
       \erf\left(\frac{t}{\sqrt{2} \sigma}\right),
\end{align} 
where the error function is defined as: 
\begin{equation*}
   \erf(t) = \frac{2}{\sqrt{\pi}} 
             \int_{0}^{t} \exp(-u^2) \dd u
   \quad\forall\, t\in\real.
\end{equation*}
\end{lemma}
\begin{proof}
\end{proof}
The above lemma is proved in \cite{voronin2014convolution}.
Next, using the fact that the error function $\erf(t) = \frac{2}{\sqrt{\pi}} \int_{0}^t \exp(-s^2) \dd s$ satisfies the bounds:
\begin{equation}
\label{eq:erf_bounds}
   \bigl(1 - \exp(-t^2)\bigr)^\frac{1}{2} \leq 
   \erf(t) \leq 
   \bigl(1 - \exp(-2 t^2)\bigr)^\frac{1}{2},
      \quad\forall\, t\geq0,
\end{equation}
the following convergence result can be established:
\begin{proposition}
\label{prop:conv}
Let $g(t):=|t|$ for all $t\in\real$, 
and let the function $\phi_\sigma:=K_\sigma*g$
be defined as in \eqref{eq:phisigma}, for all $\sigma>0$.
Then:
\begin{equation*}
\lim_{\sigma \to 0} \left\|\phi_{\sigma} - g\right\|_{L^1} = 0.
\end{equation*}
\end{proposition}
The proof is again given in \cite{voronin2014convolution}.
While $g=|\cdot| \not \in L^1$ (since $g(t) \to \infty$ as 
$t \to \infty$), the approximation in the $L^1$ norm still holds. It is likely 
that the convolution approximation converges to $g$ in the $L^1$ norm for a variety of non-smooth coercive functions $g$, 
not just for $g(t) = |t|$. 

\subsection{Approximation at zero}
Note from \eqref{eq:x_k_approx1} that 
while the approximation $\phi_{\sigma}(t) = K_\sigma*|t|$ is indeed smooth,
it is positive on $\real$ and in particular $(K_\sigma*|\cdot|)(0)=\sqrt{\frac{2}{\pi}} \sigma>0$, 
although $(K_\sigma*|\cdot|)(0)$ does go to zero as $\sigma \to 0$. 
To address this, we can use different approximations based on $\phi_{\sigma}(t)$ which are zero 
at zero. Below, we describe several different alternatives which are possible. 
The first is formed by subtracting the value at $0$: 
\begin{equation}
\label{eq:x_k_approx2}
   \tilde{\phi}_{\sigma}(t) = \phi_\sigma(t)-\phi_\sigma(0)
   \,=\,
   t \erf \left( \frac{t}{\sqrt{2} \sigma} \right) 
   + \sqrt{\frac{2}{\pi}} \sigma \exp\left(\frac{-t^2}{2 \sigma^2}\right)
   - \sqrt{\frac{2}{\pi}} \sigma .
\end{equation}
An alternative is to use $\tilde{\phi^{(2)}}_{\sigma}(t) = \phi_\sigma(t) - \sqrt{\frac{2}{\pi}} \sigma \exp\left(-t^2\right)$ 
where the subtracted term decreases in magnitude as $t$ becomes larger and only has much effect for $t$ close to zero. 
We could also simply drop the second term of $\phi_{\sigma}(t)$ to get:
\begin{equation}
\label{eq:x_k_approx3}
   \hat{\phi}_{\sigma}(t) = \phi_\sigma(t) - \sqrt{\frac{2}{\pi}} \sigma \exp\left(\frac{-t^2}{2 \sigma^2}\right)
   \,=\,
   t \erf \left( \frac{t}{\sqrt{2} \sigma} \right) 
\end{equation}
which is zero when $t = 0$. The behavior is plotted in Figure \ref{fig:Ksigma_plot}.

\subsection{Gradient Computations and Algorithms}

We now discuss algorithms for the approximate minimization of \eqref{eq:lp_funct} using the ideas we have developed. 
In particular, we will focus on the case $l=2$ resulting in the one parameter functional: 
\begin{equation}
\label{eq:lp_funct2}
   \tilde{F}_p(x)
   =
   \|Ax - b\|^2_2 
     + \lambda \left( \sum_{k=1}^n |x_k|^p \right).
\end{equation}
We obtain the smooth approximation functional to $\tilde{F}_p(x)$:
\begin{equation}
\label{eq:approx_lp}
\begin{array}{rcl}
   H_{p,\sigma}(x) 
   &:=&\displaystyle
   \|Ax - b\|^2_2 + \lambda \left( \sum_{k=1}^n 
      \phi_\sigma(x_k)^p \right) 
\\
   &=&\displaystyle
   \|Ax - b\|^2_2 + \lambda \left( \sum_{k=1}^n 
      \left( x_k \erf \left( \frac{x_k}{\sqrt{2} \sigma} \right) 
      + \sqrt{\frac{2}{\pi}} \sigma 
      \exp\left(\frac{-x_k^2}{2 \sigma^2}\right)\right)^p \right) .
\end{array}
\end{equation}
which is of similar form considered in detail in \cite{voronin2014convolution}.
To compute the gradient of the smooth functional it is enough to consider 
the function $G_p(x) = \displaystyle\sum_{k=1}^n \phi_{\sigma}(x_k)^p$. Taking the derivative 
with respect to $x_j$ yields:
\begin{equation*}
\frac{\partial}{\partial x_j} \phi_{\sigma}(x_j)^p = p \phi_{\sigma}(x_j)^{p-1} \phi_{\sigma}^{\prime}(x_j) = p \phi_{\sigma}(x_j)^{p-1} \erf\left( \frac{x_j}{\sqrt{2} \sigma} \right).
\end{equation*}
Taking the second derivative yields: 
\begin{equation*}
\frac{\partial^2}{\partial x_i \partial x_j} G_p(x) = 0, 
\end{equation*}
when $i \neq j$ and when $i=j$ we have:
\begin{eqnarray*}
\frac{\partial^2}{\partial x_j^2} G_p(x) &=& \frac{\partial}{\partial x_j} \left\{ \left[ p \phi_{\sigma}(x_j)^{p-1} \right] \left[ \erf\left( \frac{x_j}{\sqrt{2} \sigma} \right) \right] \right\} \\
&=& p(p-1) \phi_{\sigma}(x_j)^{p-2} \erf^2\left( \frac{x_j}{\sqrt{2} \sigma} \right) + \left[ p \phi_{\sigma}(x_j)^{p-1} \right] \frac{2}{\sqrt{2} \sqrt{\pi} \sigma} \exp\left( -\frac{x_j^2}{2 \sigma^2} \right) .
\end{eqnarray*}
The following results follow.
\begin{lemma}
Let $H_{p,\sigma}(x)$ be as defined in \eqref{eq:approx_lp} where $p>0$ and $\sigma > 0$.
Then the gradient is given by:
\begin{equation}
\label{eq:lp_gradient}
   \nabla {H_{p,\sigma}}(x) = 
   2 A^T (Ax - b) + \lambda p \bigl(\vec{v}(x)\bigr), 
\end{equation}
and the Hessian is given by:
\begin{equation}
\label{eq:lp_hessian}
   \nabla^2 H_{p,\sigma}(x) = 2 A^T A + \lambda p \Diag\bigl(\vec{w}(x)\bigr),
\end{equation}
where the functions $\vec{v}: \real^n\to\real^n$ and $\vec{w}:\real^n\to\real^n$ are defined for all $x\in\real^n$:
\begin{small}
\begin{align*}
   \vec{v}(x):= \left\{ v(x_j) \right\}_{j=1}^n = &\
\left\{ \phi_{\sigma}(x_j)^{p-1} 
   \erf\left(\frac{x_j}{\sqrt{2} \sigma}\right) \right\}_{j=1}^n \\
   \vec{w}(x):= \left\{ w(x_j) \right\}_{j=1}^n = &\ 
   \left\{
		(p-1) \phi_{\sigma}(x_j)^{p-2} \erf^2\left( \frac{x_j}{\sqrt{2} \sigma} \right) + \left[ \phi_{\sigma}(x_j)^{p-1} \right] \sqrt{\frac{2}{\pi \sigma^2}} \exp\left( - \frac{x_j^2}{2 \sigma^2} \right)
   \right\}_{j=1}^n.
\end{align*}
\end{small}
\end{lemma}

Given $H_{p,\sigma}(x) \approx \tilde{F}_p(x)$ and $\nabla H_{p,\sigma}(x)$, we can apply a number of gradient based methods for the minimization of $H_{p,\sigma}(x)$ (and hence for the approximate minimization 
of $\tilde{F}_p(x)$), which take the following general form:

\begin{algorithm}[ht!]
\SetKwInOut{Input}{Input}
\SetKwInOut{Output}{Output}
\caption{Generic Gradient Method for finding $\arg\min H_{p,\sigma}(x)$.}
\label{algo:gradient_method}
\BlankLine
Pick an initial point $x^0$\;
\For{$n=0,1,\ldots$,\texttt{\textup{ maxiter}}}{
Compute search direction $s^n$ based on gradient 
$\nabla H_{p,\sigma}(x^n)$. \;
Compute step size parameter $\mu$ via line search. \;
Update the iterate: $x^{n+1} = x^n + \mu s^n$. \;
Check if the termination conditions are met. \;
}
Record final solution: $\bar{x} = x^{n+1}$. \;
\end{algorithm}
Note that in the case of $p<1$, the functional $F_p(x)$ is not convex, so such an algorithm may 
not converge to the global minimum in that case. The generic algorithm above depends on the choice 
of search direction $s^n$, which is based on the gradient, and the line search, 
which can be performed in several different ways. 

\noindent
\subsection{Line Search Techniques}
Gradient based algorithms differ based on 
the choice of search direction vector $s^n$ and 
line search techniques for parameter $\mu$. 
In this section we describe some suitable line search 
techniques. 
Given the current iterate $x^n$ and search direction $s^n$, 
we would like to choose $\mu$ so that:
\begin{equation*} 
   H_{p,\sigma}(x^{n+1})
   = H_{p,\sigma}(x^n + \mu s^n)
   \leq H_{p,\sigma}(x^n),
\end{equation*}
where $\mu>0$ is a scalar which measures how long along 
the search direction we advance from the previous iterate. 
Ideally, we would 
like a strict inequality and the functional value to decrease. 
Exact line search would solve the single variable 
minimization problem:
\begin{equation*}
   \bar{\mu} = \arg\min_{\mu} H_{p,\sigma}(x^n + \mu s^n) .
\end{equation*}
The first order necessary optimality condition
(i.e., $\nabla H_{p,\sigma}(x + \mu s)^T s=0$) can be used to find a 
candidate value for $\mu$, but it is not easy to solve the gradient equation. 
Instead, using the second order Taylor approximation of $n(t):=H_{p,\sigma}(x+t s)$ at any given $x,s\in\real^n$, 
we have that 
\begin{equation}
\label{eq:taylor_expansion_of_H}
   n^{\prime}(t)=n^{\prime}(0)+t n^{\prime \prime}(0)+o(t) \approx n^{\prime}(0)+t n^{\prime \prime}(0)
\end{equation}
using basic matrix calculus:
\begin{eqnarray*}
   n^{\prime}(t) 
   &=& 
   \left( \nabla H_{p,\sigma}(x+ts) \right)^T s 
   \implies 
   n^{\prime}(0) = \nabla H_{p,\sigma}(x)^T s 
\\
   n^{\prime \prime}(t) 
   &=& \left[ \left( \nabla^2 H_{p,\sigma}(x + ts) \right)^T s \right]^T s =  
   s^T \nabla^2 H_{p,\sigma}(x + ts) s 
   \implies 
   n^{\prime \prime}(0) = s^T \nabla^2 H_{p,\sigma}(x) s,
\end{eqnarray*}
we get that $n^{\prime}(0)+\mu n^{\prime \prime}(0)=0$ if and only if
\begin{equation}
\label{eq:mu_approx_sol}
\mu = -\frac{\nabla H_{p,\sigma}(x)^T s}{s^T \nabla^2 H_{p,\sigma}(x) s}.
\end{equation}

An alternative approach is to use a backtracking line search  
to get a step size $\mu$ that satisfies one or two of 
the Wolfe conditions \cite{nocedal_wright_opt}. 
This update scheme can be slow since several evaluations 
of $H_{p,\sigma}(x)$ may be necessary, which are relatively expensive when the 
dimension $n$ is large. The Wolfe condition scheme also necessitates the 
choice of further parameters. 

\vspace{3.mm}
\subsection{Nonlinear Conjugate Gradient Algorithm}
We now present the conjugate gradient scheme in Algorithm \ref{algo:nonlincg}, 
which can be used for sparsity constrained regularization. In the basic (steepest) descent algorithm, 
we simply take the negative of the gradient as the search direction. For nonlinear conjugate gradient methods, 
several different search direction updates are possible. 
We find that the Polak-Ribi\`{e}re scheme often offers good performance 
\cite{MR0255025,Polyak196994,shewchukCG}. 
In this scheme, we set the initial search direction $s^0$  
to the negative gradient, as in steepest descent, but then do a more 
complicated update involving the gradient at the current and previous steps:
\begin{eqnarray*}
   \beta^{n+1} 
   &=& 
   \max\left\{ \frac{\nabla H_{p,\sigma_n}(x^{n+1})^T 
      \left(\nabla H_{p,\sigma_n}(x^{n+1}) - \nabla H_{p,\sigma_n}(x^n)\right)}
      {\nabla H_{p,\sigma_n}(x^n)^T \nabla H_{p,\sigma_n}(x^n)}, 0 \right\} ,
\\
   s^{n+1} 
   &=& -\nabla H_{p,\sigma_n}(x^{n+1}) + \beta^{n+1} s^n.
\end{eqnarray*}

One extra step we introduce in Algorithm 
\ref{algo:nonlincg} is a thresholding 
which sets small components to zero. That is, at the end of each iteration, we 
retain only a portion of the largest coefficients. This is necessary, 
as otherwise the solution we recover will contain many small noisy components 
and will not be sparse. In our numerical experiments, we found that soft 
thresholding works well when $p=1$ and that hard thresholding works better   
when $p<1$. 
The component-wise soft and hard thresholding functions with parameter 
$\lambda>0$ are given by: 
\begin{equation}
\label{eq:thresholding}
   \left(\mathbb{S}_{\lambda }(x)\right)_k = \left\{
   \begin{array}{ll}
       x_k - \lambda, & \hbox{$x_k > \lambda$} \\
       0, & \hbox{$-\lambda \le x_k \le \lambda$;} \\
       x_k + \lambda, & \hbox{$x_k < -\lambda$ } \\
   \end{array}
   \right. \quad \left(\mathbb{H}_{\lambda }(x)\right)_k = \left\{
   \begin{array}{ll}
       x_k, & \hbox{$|x_k| > \lambda$} \\
       0, & \hbox{$-\lambda \le x_k \le \lambda$}
   \end{array},
   \right. 
\quad\forall\, x\in\real^n.
\end{equation}
For $p=1$, an alternative to thresholding at each iteration at $\lambda$ is to 
use the optimality condition of the $F_1(x)$ functional \cite{ingrid_thresholding1}. 
After each iteration (or after a block of iterations), we can evaluate the vector  
\begin{equation}
\label{eq:v_thresholding}
v^n = A^T (b - A x^n) . 
\end{equation}
We then set the components (indexed by $k$) of the current solution vector 
$x^n$ to zero for indices $k$ for which $|v^n_k| \leq \frac{\lambda}{2}$. 

Note that after each iteration, we also vary the parameter $\sigma$ in the 
approximating function to the absolute value $\phi_\sigma$, 
starting with $\sigma$ relatively far from zero at the first iteration 
and decreasing towards $0$ as we approach the iteration limit. 
The decrease can be controlled by a parameter $\alpha\in(0,1)$ 
so that $\sigma_{n+1} = \alpha \sigma_n$. 
The choice $\alpha = 0.8$ worked well in our experiments. 
Comments on the computational cost relative to the FISTA algorithm are discussed in 
\cite{voronin2014convolution}, where it is shown that the most expensive matrix-vector 
multiplication operations are present in both algorithms. On the other hand, the overhead 
with using CG iterations is significant and each iteration does take more time than 
that of a thresholding method. The algorithm is designed to be run for a small number of 
iterations.

In Algorithm \ref{algo:nonlincg}, we present a nonlinear Polak-Ribi\`{e}re conjugate 
gradient scheme to approximately minimize $\tilde{F}_p$ \cite{MR0255025,Polyak196994,shewchukCG}.
The function $\mathtt{Threshold}(\cdot,\tau)$ in the algorithms which enforces sparsity refers to either one of the two thresholding functions defined in \eqref{eq:thresholding} or to the strategy using the $v^n$ vector in \eqref{eq:v_thresholding}.
The update rule for $\sigma$ can be varied. In particular, it can again be tied to the distance between two successive iterates
(e.g. $\sigma_{n+1} = \min\left(\sigma_0, \alpha*\|x^{n+1} - x^n\|\right)$), although care must be taken not to make 
$\sigma$ too small, which has the effect of introducing a nearly sharp corner. 
Another possibility, given access to both gradient and Hessian, is to use a higher 
order root finding method, such as Newton's method \cite{nocedal_wright_opt} as discussed in 
\cite{voronin2014convolution}.

\begin{algorithm}[!ht]
\SetKwInOut{Input}{Input}
\SetKwInOut{Output}{Output}
\caption{CONV CG Algorithm
\label{algo:nonlincg}}
\Input{An $m\times n$ matrix $A$, 
an initial guess $n \times 1$ vector $x^0$, 
a parameter $\tau < \|A^T b\|_\infty$, 
a parameter $p\in[1,2]$, 
a parameter $\sigma_0>0$, 
a parameter $0 < \alpha < 1$, 
the maximum number of iterations $N$, 
and a routine to evaluate the gradient $\nabla H_{p,\sigma}(x)$ (and possibly 
the Hessian $\nabla^2 H_{p,\sigma}(x)$ depending on choice of line search method).}
\Output{A vector $\bar{x}$, 
close to either the global or local minimum of $\tilde{F}_{p}(x)$, 
depending on choice of $p$.}
\BlankLine
$s^0 = - \nabla H_{p,\sigma_0}(x^0)$ \;
\For{$n=0,1,\ldots$,N}{
use line search to find $\mu>0$\;
$x^{n+1} = \mathtt{Threshold}(x^n + \mu s^n, \tau)$ \;
$\beta^{n+1} = 
\max\left\{ \frac{\nabla H_{p,\sigma_n}
(x^{n+1})^T (\nabla H_{p,\sigma_n}(x^{n+1}) - \nabla H_{p,\sigma_n}(x^n))}
{\nabla H_{p,\sigma_n}(x^n)^T \nabla H_{p,\sigma_n}(x^n)}, 0 \right\}$ \;
$s^{n+1} = -\nabla H_{p,\sigma_n}(x^{n+1}) + \beta^{n+1} s^n$ \;
$\sigma_{n+1} = \alpha \sigma_n$ \; 
}
$\bar{x} = x^{n+1}$\;
\end{algorithm}

\subsection{Application to generalized residual penalty and wavelet representations}

First, we comment on the application of the CONV CG method to \eqref{eq:lp_funct}. In this case, the change of variables 
$y = Ax - b$ gives the constrained minimization problem:
\begin{equation*}
\min_{y,x} \left\{ \|y\|_l^l + \lambda \|x\|_p^p \right\} \quad \mbox{s.t.} \quad y = Ax - b .
\end{equation*}
This can be accomplished via e.g. an alternative variable minimization scheme combined 
with the Lagrange multiplier method, in which case we get the minimization 
problem:
\begin{equation*}
\min_{y,x,s} \left\{  \|y\|_l^l + \lambda \|x\|_p^p + s^T(Ax - b - y)  \right\} 
\end{equation*}
where $s$ is a vector of Lagrange multipliers.
We can use the alternate minimization method to minimize with respect to each variable in a loop. 
Let us now assume that $l \in (1,2)$ and that all $y_i \neq 0$. In this case, minimizing with respect to $y$ by setting the 
gradient to zero:
\begin{equation*}
l \left\{y_i^{l-1}\right\}_{i=1}^m - s = 0 .
\end{equation*}
Next, for minimizing with respect to $x$, we will use the convolution based approximation, 
since some entries of $x$ can indeed be zero. We get, for each component:
\begin{equation*}
\frac{\partial}{\partial x_i} \left[ \displaystyle\sum_{k=1}^n \phi_{\sigma}(x_k)^p + s^T A x \right] = 
p \phi_{\sigma}(x_i)^{p-1} \erf\left(\frac{x_i}{\sqrt{2} \sigma}\right) + \left[A^T s\right]_i = 0,
\end{equation*}
which is a nonlinear system to be solved for each component $i$. (The derivative with respect to $s$ 
simply yields $Ax - b - y = 0$). A more efficient approach is to again assume that all $r_i = (Ax - b)_i \neq 0$ 
and utilize the same result from \cite{scales1988fast}, yielding in place of \eqref{eq:lp_gradient},
the gradient:
\begin{equation*}
\nabla {H_{l,p,\sigma}}(x) = A^T R (Ax - b) + \lambda p \bigl(\vec{v}(x)\bigr), 
\end{equation*} 
with $R = \diag(l |r_i(x)|^{l-2})$ as before. In practice, $R$ would be iteration dependent, as in the IRLS scheme.
The Hessian can also be computed using the results from \eqref{eq:scales_derivative_calc1}, 
taking care of the fact that $R$ depends on $x$. 

Secondly, we comment on the application of our methods in a transformed basis. This is useful when for example, 
we are dealing with images, which are not outright sparse, but can indeed be efficiently represented by the 
use of wavelet transforms. In this case, we would like to apply the sparse penalty to $w = W x$, where $W$ is the 
wavelet transformed matrix. While $w$ may not be sparse, we can remove many of the smaller magnitude coefficients of 
$w$, such that the resulting vector $\tilde{w}$ satisfies, $W^{-1} \tilde{w} \approx x$. 
We can consider the minimization of the functional:
\begin{equation*}
\|A x - b\|_l^l + \lambda \|w\|_p^p,
\end{equation*} 
which with the substitution $x = W^{-1} w$ becomes a function of a single variable $w$. Both algorithms can then 
be applied to this formulation, using the matrix $A W^{-1}$ (in practice, we only need to be able to apply the 
transform and not to form the $W^{-1}$ matrix explicitly). For values of $p$ closer to $1$ this formulation often yields better 
results then the inversion in the default basis for signals which are approximately sparse under $W$ in the above sense.  
For image deconvolution and other formulations, we can also introduce a smoothing term into the regularization. 
For instance, we can minimize:
\begin{equation*}
\|Ax - b\|_2^2 + \lambda_1 \|w\|_p^p + \lambda_2 \|L x\|_2^2,
\end{equation*}
with $L$ a tridiagonal ``Laplacian'' kind of matrix. We can again plug in $x = W^{-1} w$ or 
extend this using the Lagrange multiplier formulation, obtaining the optimization problem:
\begin{equation*}
\min_{x,w,s} P(x,w,s) =  \min_{x,w,s} \left\{ \|Ax - b\|_2^2 + \lambda_1 \|w\|_p^p + \lambda_2 \|L x\|_2^2 + s^T (Wx - w) \right\} .
\end{equation*}
The minimization with respect to $s$ (yielding $Wx = w$) and with respect to $x$ yielding:
\begin{equation*}
2 A^T (A x - b) + 2 \lambda_2 L^T L x + W^T s = 0, 
\end{equation*}
to be solved for $s$ are straightforward. On the other 
hand, the minimization with respect to $w$ (making use of the convolution gradient result), produces again a nonlinear system. 
For instance, taking $p=1$, $\frac{\partial P}{\partial w} = 0$ gives:
\begin{equation*}
(A W^{-1})^T (A W^{-1} w - b) + \lambda_1 \erf\left(\frac{w_j}{\sqrt{2} \sigma}\right)|_{j=1}^n - s = 0
\end{equation*} 
which is a nonlinear system for each $j$. When the dimensionality ($n$) is large, solving such systems at each 
iteration is very expensive. We aim to investigate more efficient formulations for such problems (e.g. deconvolution) 
as part of upcoming work.

\section{Numerical Experiments}
\label{sect:numerics}

First, we aim to carry out a synthetic (2--D) seismic tomography experiment to quantify the usefulness of the functional \eqref{eq:lp_funct} which we consider.
In particular, we will use our IRLS CG algorithm to approximately minimize this functional with different choices of $l$.

We take a linear, tomographic system $Ax = b$, where the model parameters, $x$, are shear-wave velocity perturbations (dlnVs) and the model parametrization consists of $n=1024$ square-pixels (see Fig.\,\ref{FigTomo}(a)).
In the framework of ray theory, data represent onset time-residuals of direct \textit{S} waves, whose ray paths are straight lines from one black dot to another (see Fig.\,\ref{FigTomo}(a)).
The total number of data that we consider is $m=400$ (i.e., $m\ll n$).
Each element $A_{ij}$ of the sensitivity matrix represents the length of the $i$-th ray inside the $j$-th pixel.

For a given input, \textit{true} model, $x^{\text{true}}$, the noisy (outlier-free) data set is computed as: $b \leftarrow Ax^{\text{true}} + n$, for realistic, random noise $n$ (see Fig.\,\ref{FigTomo}(b)).
Here, we consider the \textquoteleft checkerboard' true model displayed in Fig.\,\ref{FigTomo}(c).
The damped least-squares (DLS) model solution, obtained from LSQR inversion of the (outlier-free) data set $b$ is shown in Fig.\,\ref{FigTomo}(d).

Let us consider the previous (outlier-free) data set, $b$, to which we add some \textquoteleft outlier' time-residuals, $n^{\text{outliers}}$, that is: $b^{\text{outliers}} \leftarrow b + n^{\text{outliers}}$ (see Fig.\,\ref{FigTomo}(e)).
The DLS solution, obtained from LSQR inversion of the outlier data set, $b^{\text{outliers}}$, is then displayed in Fig.\,\ref{FigTomo}(f).
Moreover, the corresponding solutions obtained from the same outlier data set, $b^{\text{outliers}}$, using our IRLS CG scheme with $l=1.0$ and $l=1.8$ are shown in Figs.\,\ref{FigTomo}(g) and (h), respectively.

We see that: 1) DLS is fine for outlier-free data, but not so for outlier data; 2) Our IRLS CG algorithm proves to give better results than DLS for outlier data -- when using $l$ values close to unity. 
That is, keeping $l$ close to $2$ gives a big effect in the solution, due to the inclusion of outlier time-residuals.
On the other hand, $l$ closer to $1$ imposes a bigger penalty on the outlier residual terms and produces a solution closer to the (DLS) case of outlier-free data.

As this synthetically constructed example shows, the ability to control the $l$ parameter in the functional allows to mitigate the effects of outlier residuals in the data.
As a remark, automatically removing outliers (e.g., related to mis-picking of seismic phases) in massive data sets may be a very difficult task, so that most data sets in (geo)physics shall come with outliers.

\begin{figure}[hb!]
\centering
 \noindent \includegraphics[width=1.0\textwidth,keepaspectratio=true]{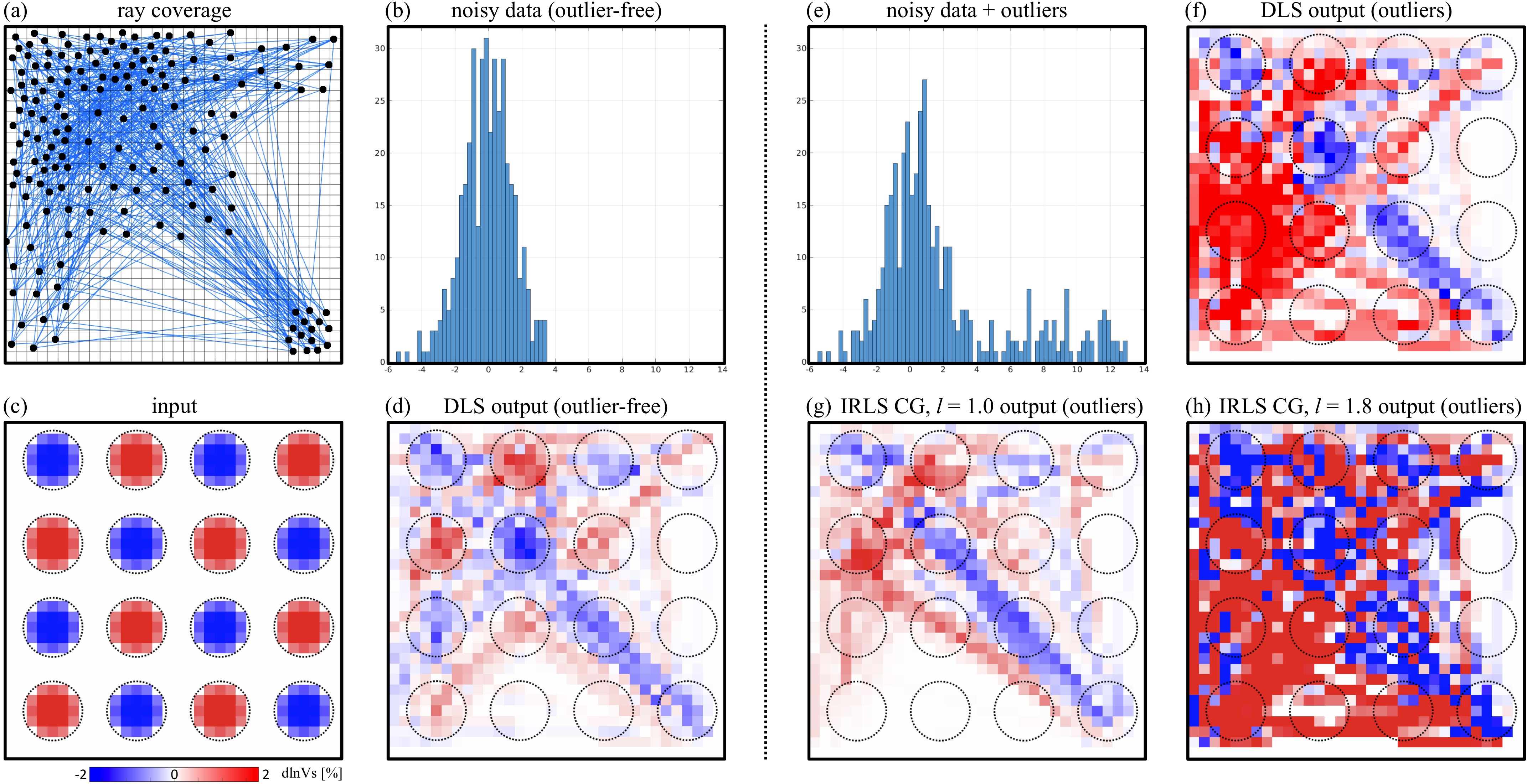}
\caption{
Tomographic experiment --
(a) Ray coverage and model parametrization (square-pixels); (b) Histogram of the outlier-free data set, $b$; (c) Input model, $x^{\text{true}}$; (d) DLS output solution for (univariant) data set $b$; (e) Histogram of the outlier data set, $b^{\text{outlier}}$; (f) IRLS CG output solution, with $l=1.0$, for data set $b^{\text{outlier}}$; (g) IRLS CG output solution, with $l=1.8$, for data set $b^{\text{outlier}}$.
}
 \label{FigTomo}
\end{figure}

In our second experiment, we compare the power of the different methods in decreasing the 
cost functional we are trying to minimize, per iteration. We run each algorithm along an L-curve, decreasing the regularization 
parameter $\lambda$ and reusing the previous solution as the initial guess at the next $\lambda$. 
Typically, we stop the procedure, either when the noise level is approximately matched (e.g. 
when $\|Ax_{\lambda} - b\|_2 \approx \|\text{noise}\|_2$, if the noise norm value is known) or at some 
optimal trade-off point along the curve. This point is often estimated by taking the region of 
maximum curvature between the terms $\log \|Ax_{\lambda} - b\|_l$ and $\|x_{\lambda}\|_p$ (or of $\|w_{\lambda}\|_p$, 
if a transformed basis is used).

In Figure \ref{fig:test2_cost}, we compare the cost reduction power of the different 
algorithms along the L curve. 
To compare with FISTA, we set $l=2, p=1$. 
We use $1000 \times 1000$ test matrices with two different rates of 
logarithmic singular value decay ($logspace(0,-0.5,k)$ and $logspace(0,-2.5,k)$ with $k = \min(m,n) = 1000$ in Octave notation). 
At each value of $\lambda$ we use $3$ iterations. From the figure,
we observe that each iteration of IRLS CG and CONV CG is more powerful than that 
of a thresholding scheme, in the sense that it results in greater functional reduction along the initial 
parts of the L curve. In particular, while each iteration of the CONV CG is more expensive than that of other schemes, 
just a few iterations of the method would be enough to yield a good warm start solution, or build an estimate of the shape of the L curve. As can be observed in 
Figure \ref{fig:test2_cost}, the effect becomes more pronounced as the matrix 
condition worsens. 

In Figure \ref{fig:test2_lcurve}, we present the results of a wavelet based reconstruction experiment 
using a CDF97 wavelet basis and a simple multi-scale model.  
In this experiment, a complete L curve is constructed. In the first row, we illustrate the L curve based reconstruction 
with the CONV CG algorithm using a well conditioned matrix $A$. Using $5$ iterations at each value of 
$\lambda$ (on a logarithmic scale of $50$ values from $\frac{\|(A W^{-1})^T b\|_\infty}{1.2}$ to 
$\frac{\|(A W^{-1})^T b\|_\infty}{10^6}$), we run the CONV CG scheme to minimize 
$\|A x - b\|_2^2 + \lambda \|w\|_1$ (applying the sparsity constraint in the Wavelet basis) and 
reuse the solution as the initial guess at each next iteration. We reconstruct the solution in the original 
basis by applying the inverse transform. We plot the tradeoff curve of 
the quantities $\log \|A x_{\lambda} - b\|_2$ and $\log \|w_{\lambda}\|_1$ and a plot of the curvature of these 
quantities as a function of $\lambda$, estimated using finite differences. As we can see from the error vs parameter plot, the 
reconstruction error drops as we approach the point of maximum curvature along the L curve. 
In the subsequent two rows, we compare the performance of FISTA and CG schemes for such a reconstruction, 
using $5$ iterations at each value of $\lambda$. Now however, we use a worse conditioned 
matrix with logspaced singular values ($logspace(0,-1.5,k)$), for which the problem is more challenging. We see that the 
CONV CG algorithm produces lower percent errors and hence, a closer reconstruction for the same 
number of iterations used. This example illustrates the utility of the CONV CG scheme for 
efficient model reconstruction and for parameter (optimal $\lambda$ value) estimation.

\begin{figure*}[ht!]
\centerline{
\includegraphics[scale=0.40]{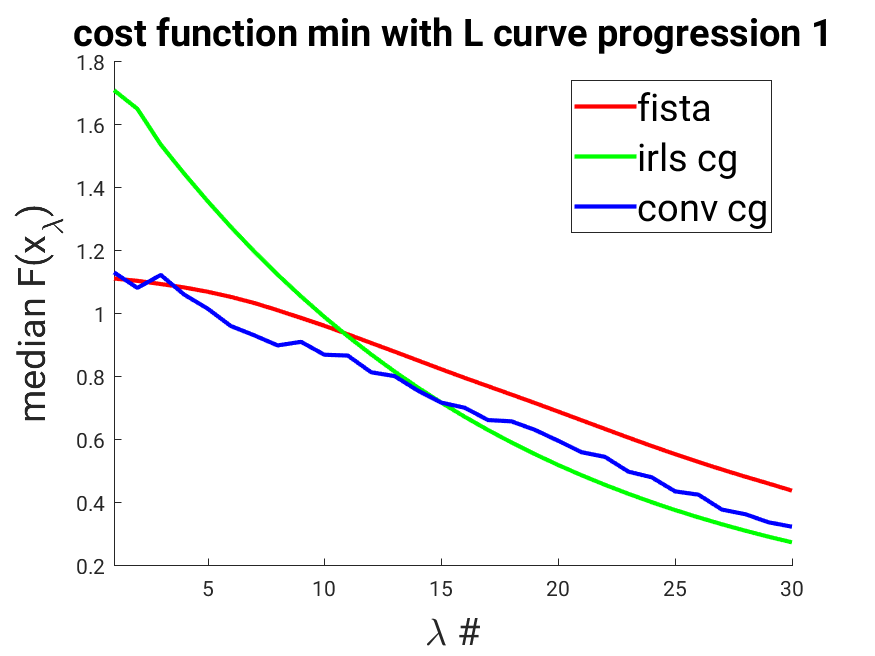}
\includegraphics[scale=0.40]{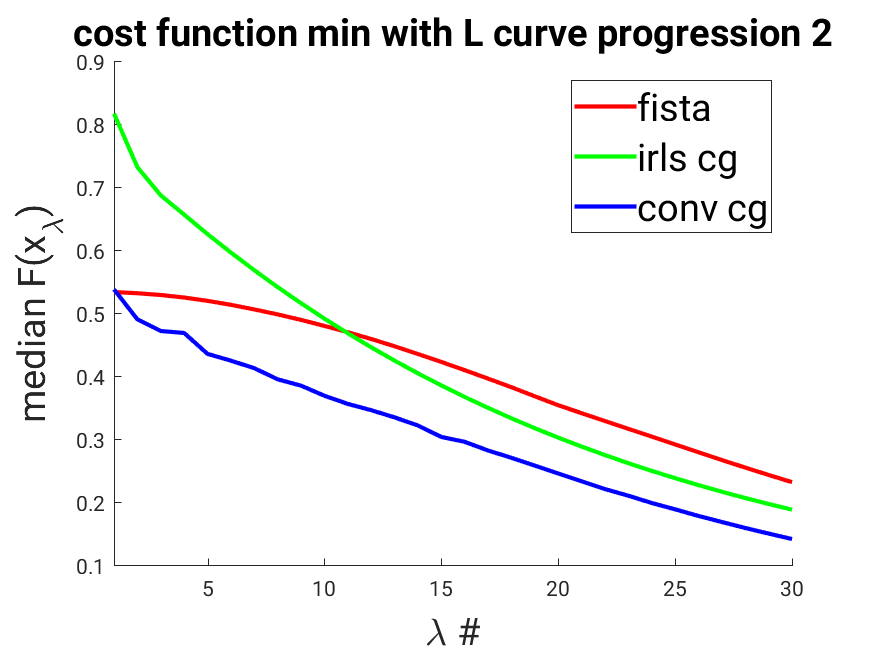}
}
\caption{Comparison of cost function reduction (over 10 trials) along L-curve progression by FISTA, IRLS CG, and CONV CG schemes over better and worse conditioned matrices. \label{fig:test2_cost}}.
\end{figure*}

\begin{figure*}[ht!]
\centerline{
\includegraphics[scale=0.25]{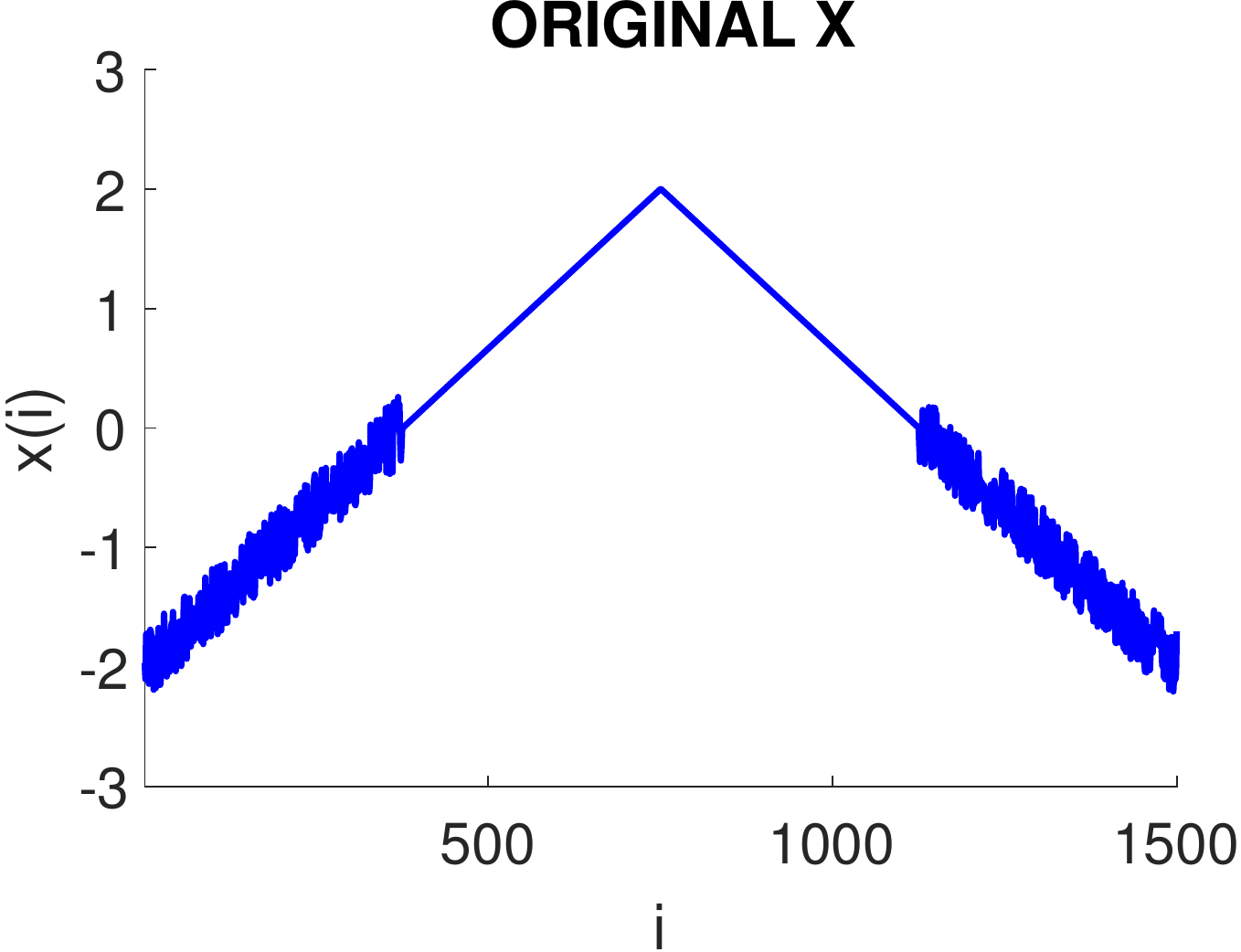}
\includegraphics[scale=0.25]{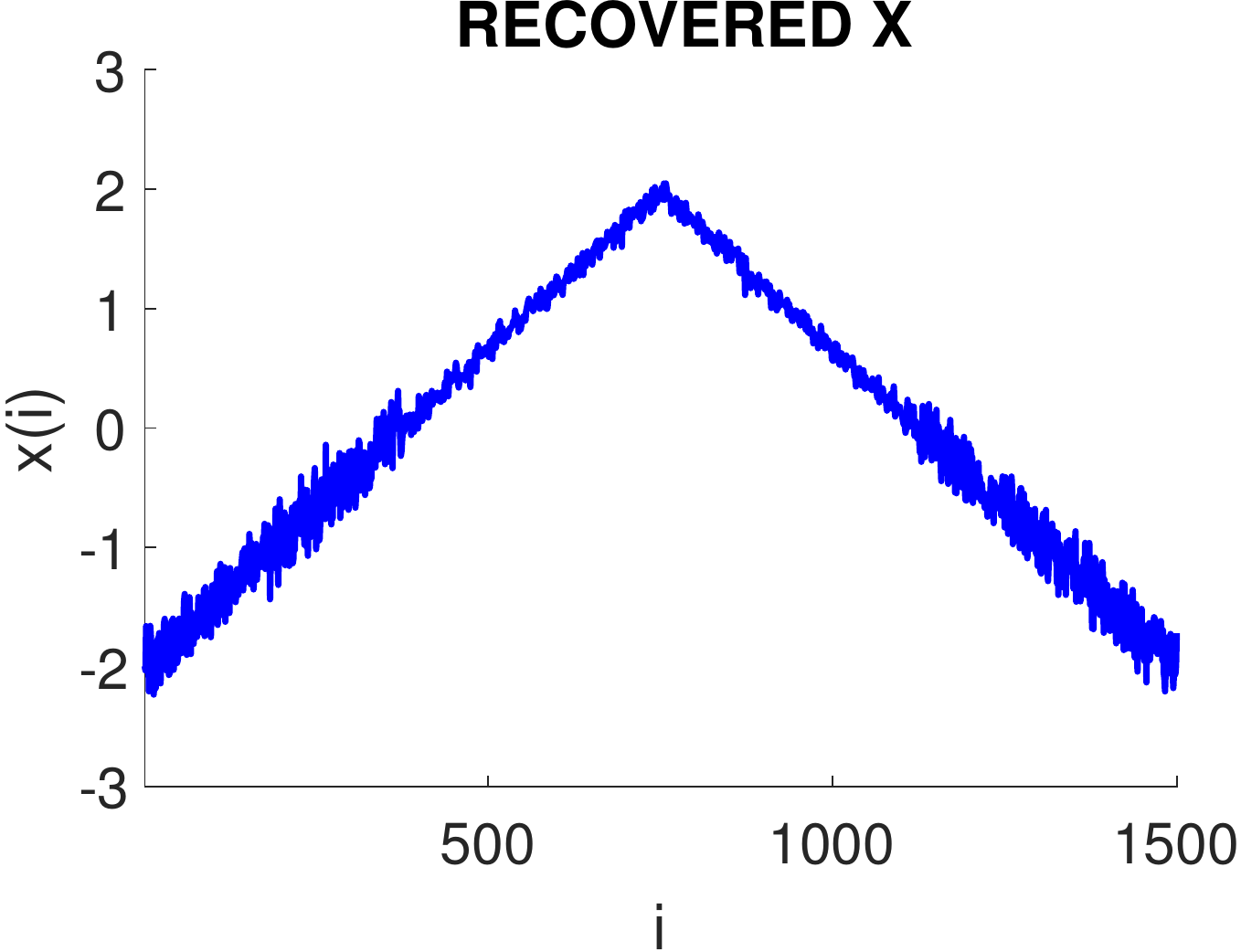}
\includegraphics[scale=0.25]{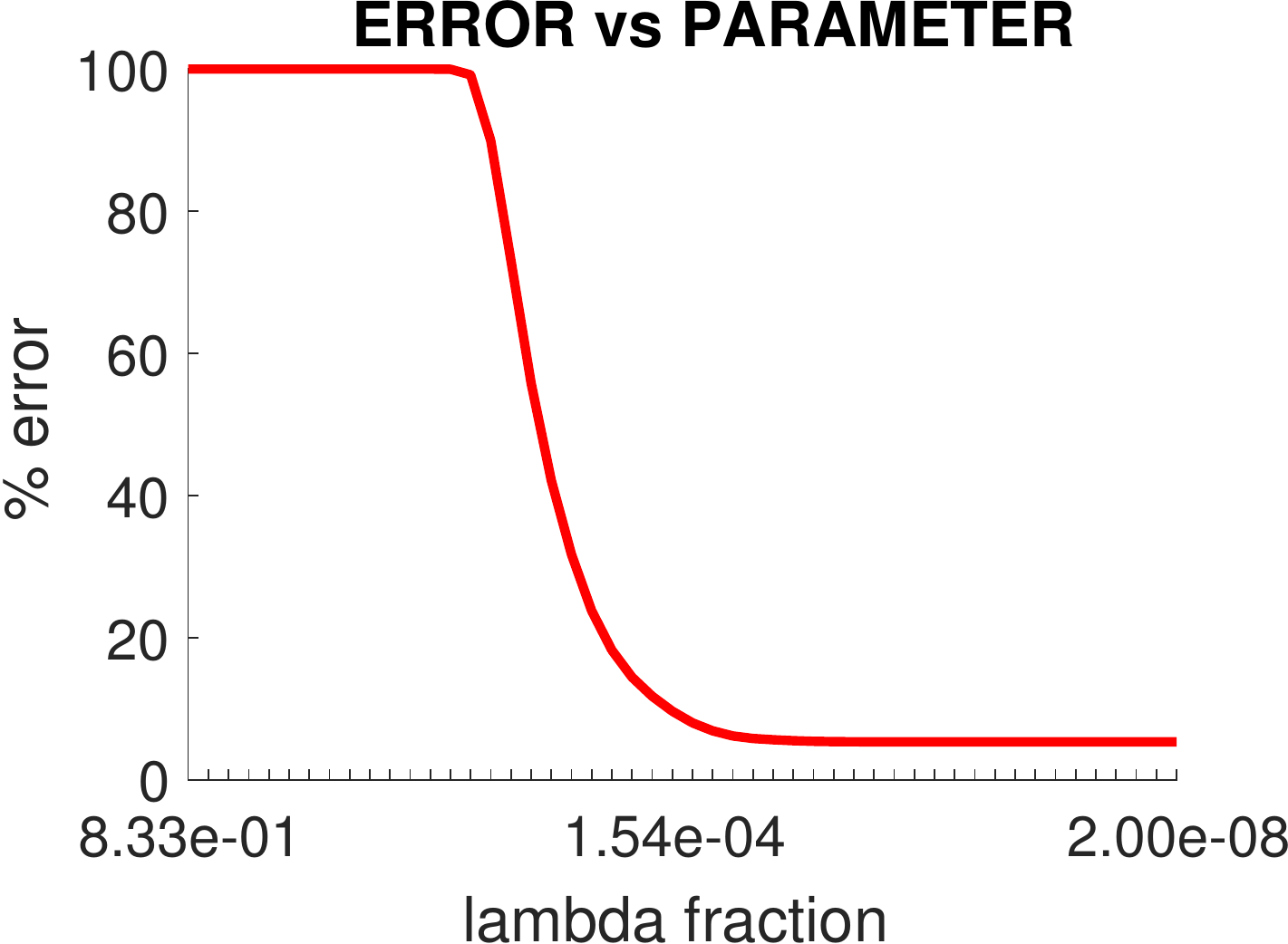}
\includegraphics[scale=0.25]{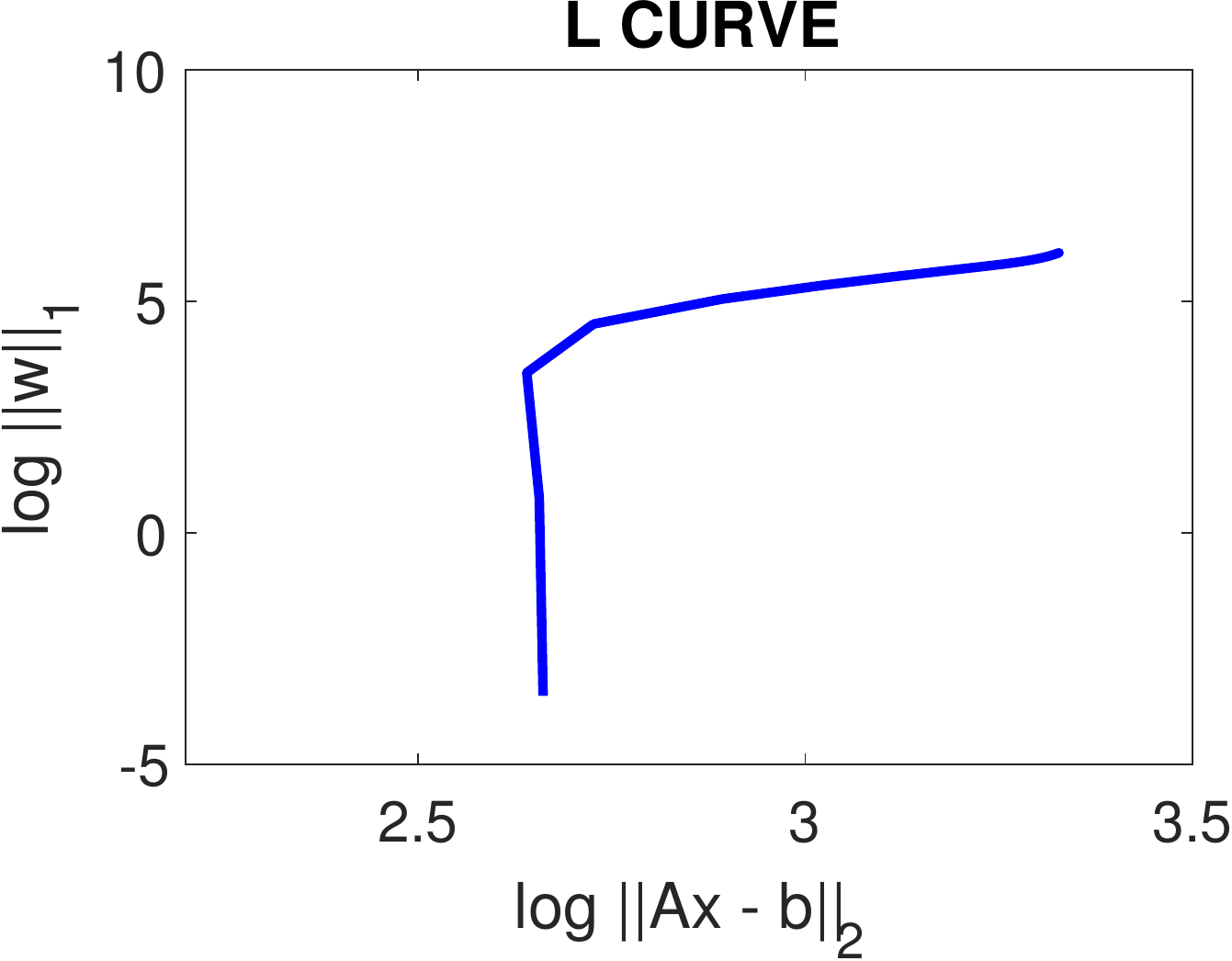}
\includegraphics[scale=0.25]{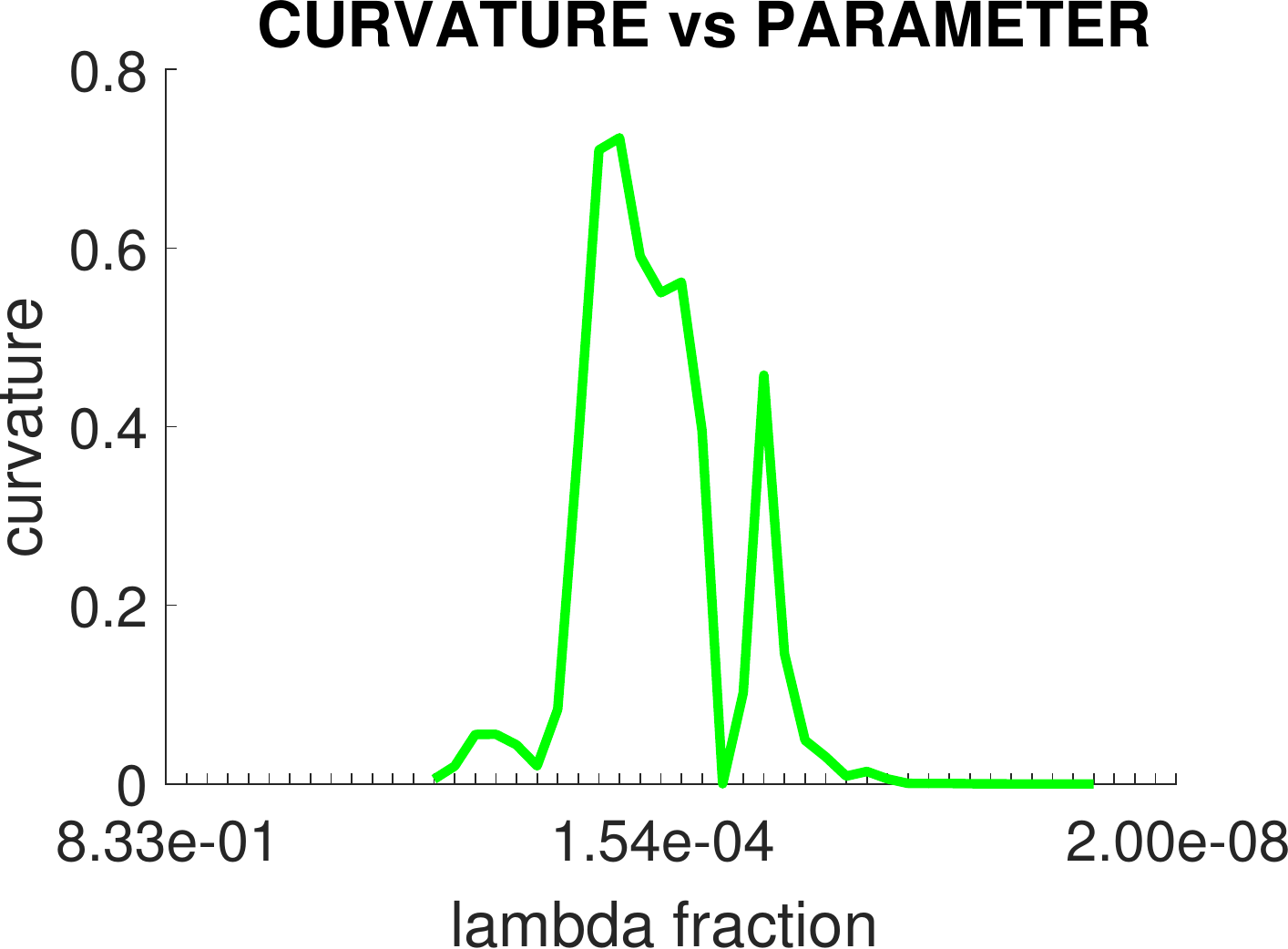}
}
\centerline{
\includegraphics[scale=0.25]{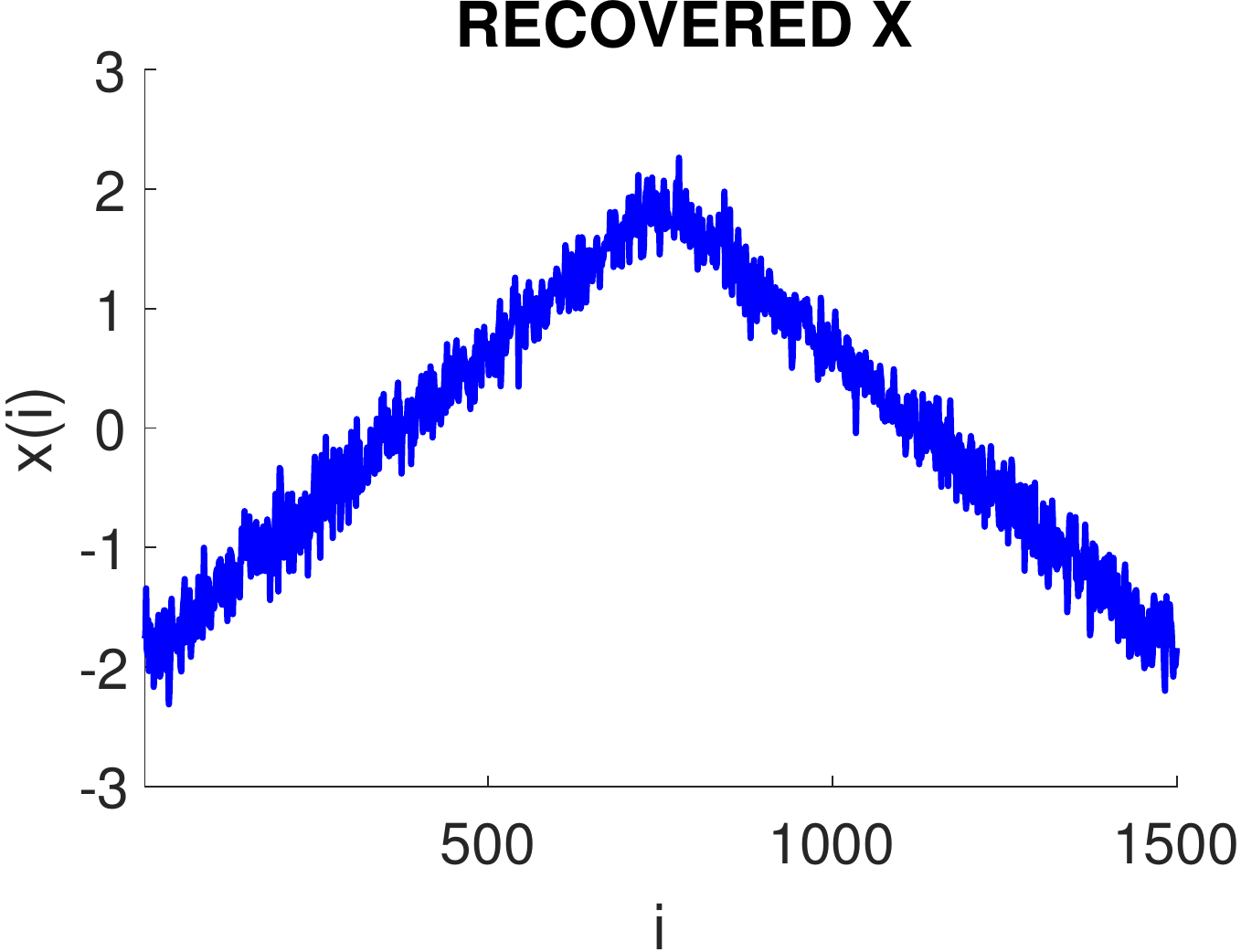}
\includegraphics[scale=0.25]{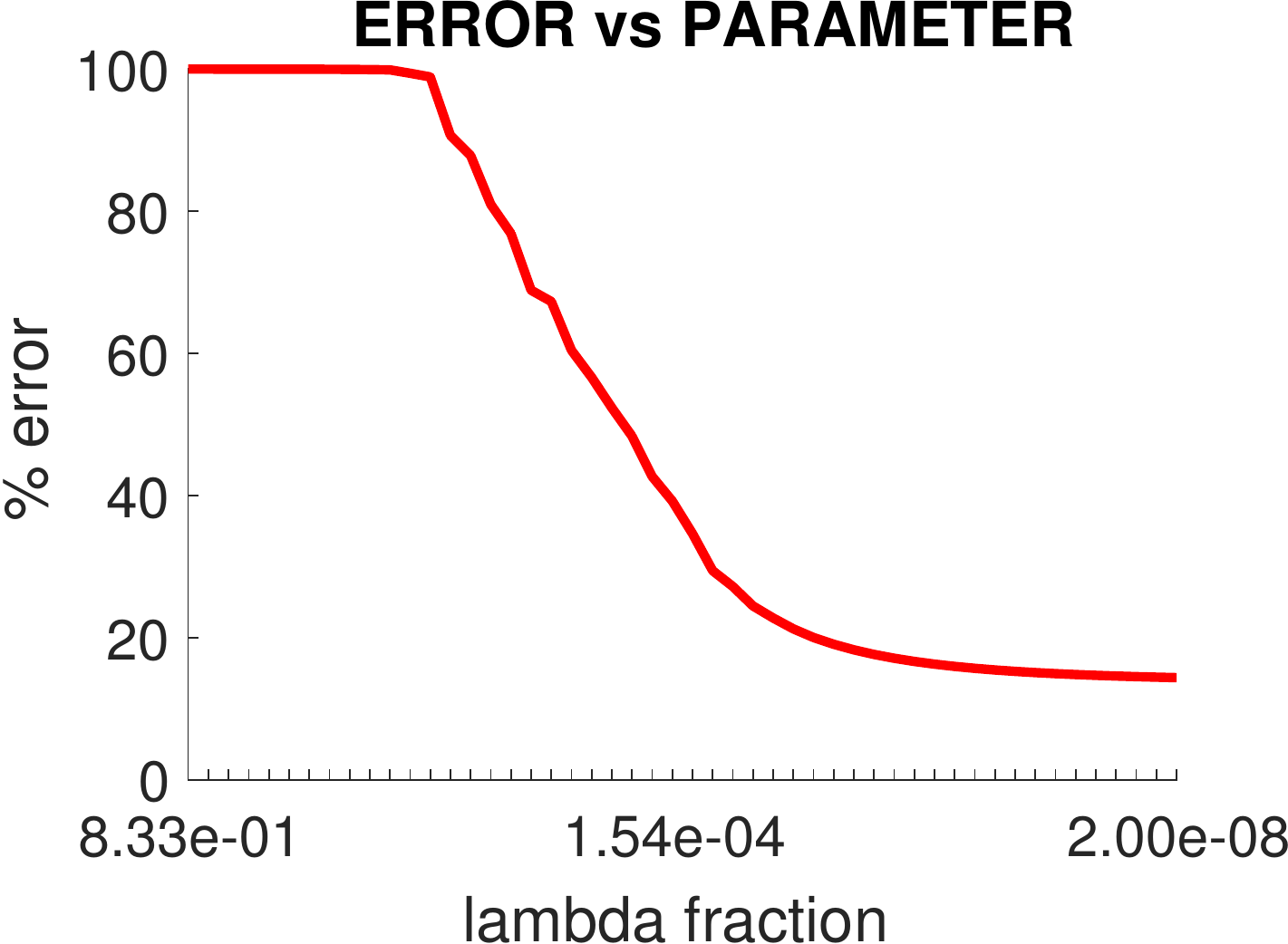}
\includegraphics[scale=0.25]{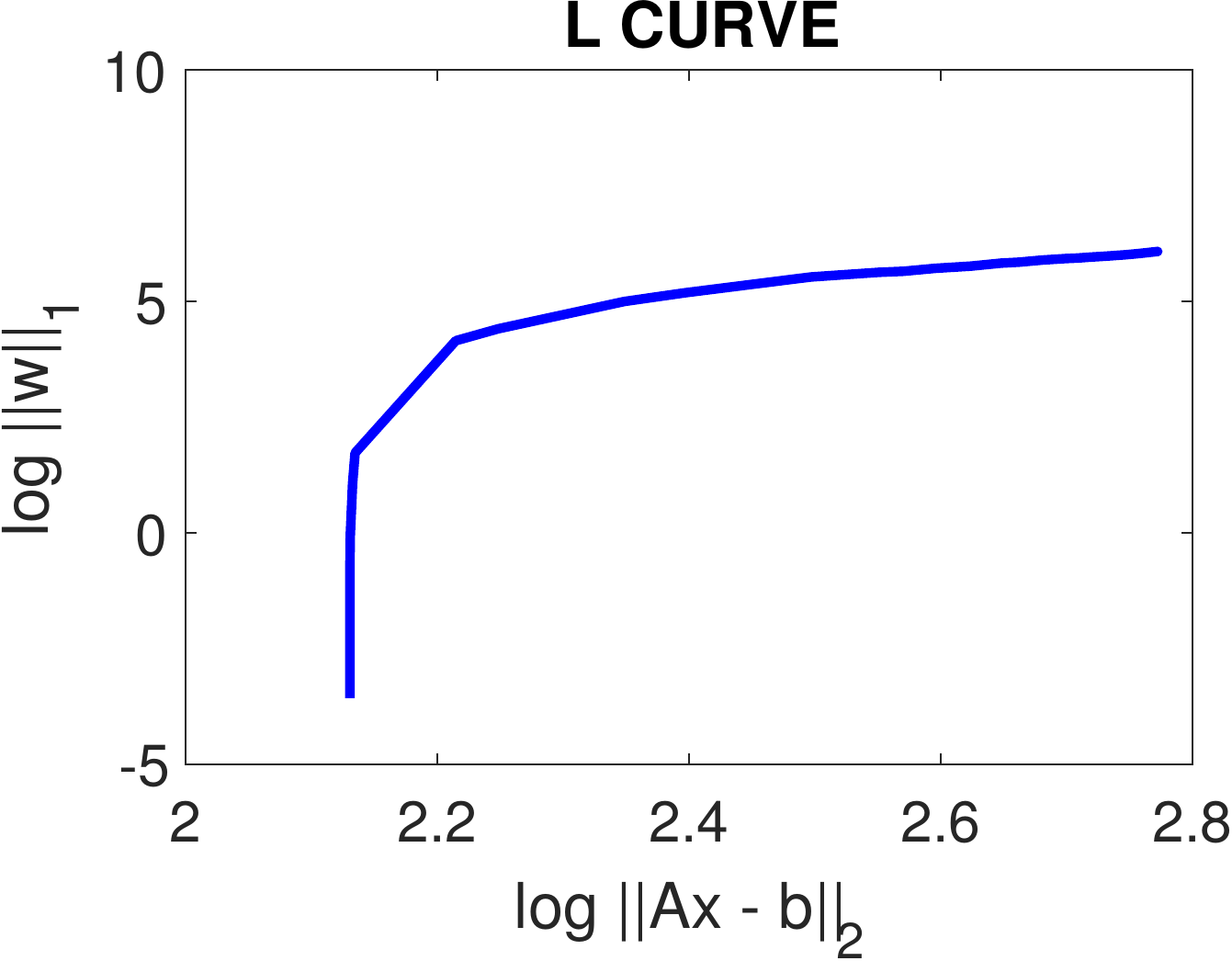}
}
\centerline{
\includegraphics[scale=0.25]{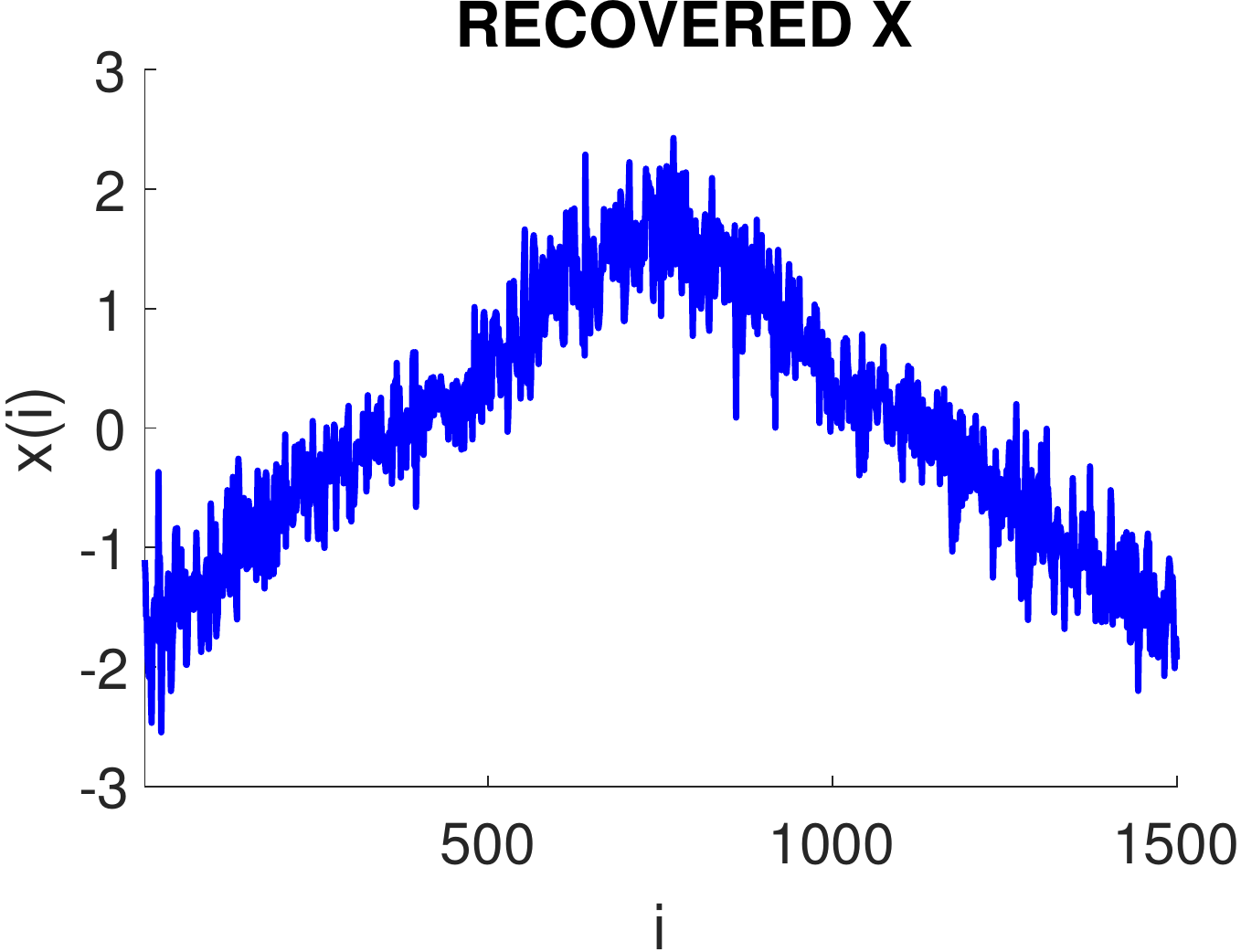}
\includegraphics[scale=0.25]{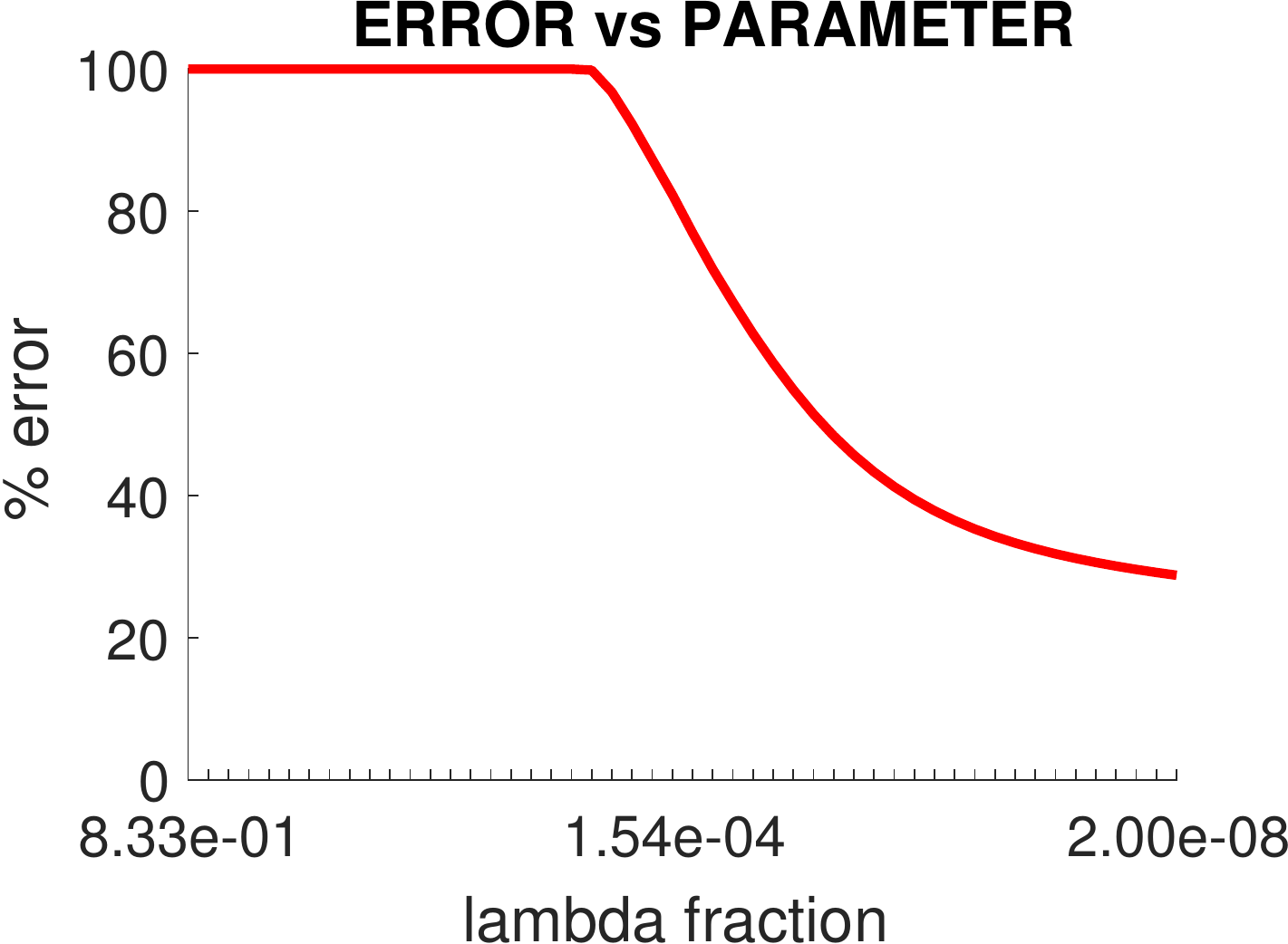}
\includegraphics[scale=0.25]{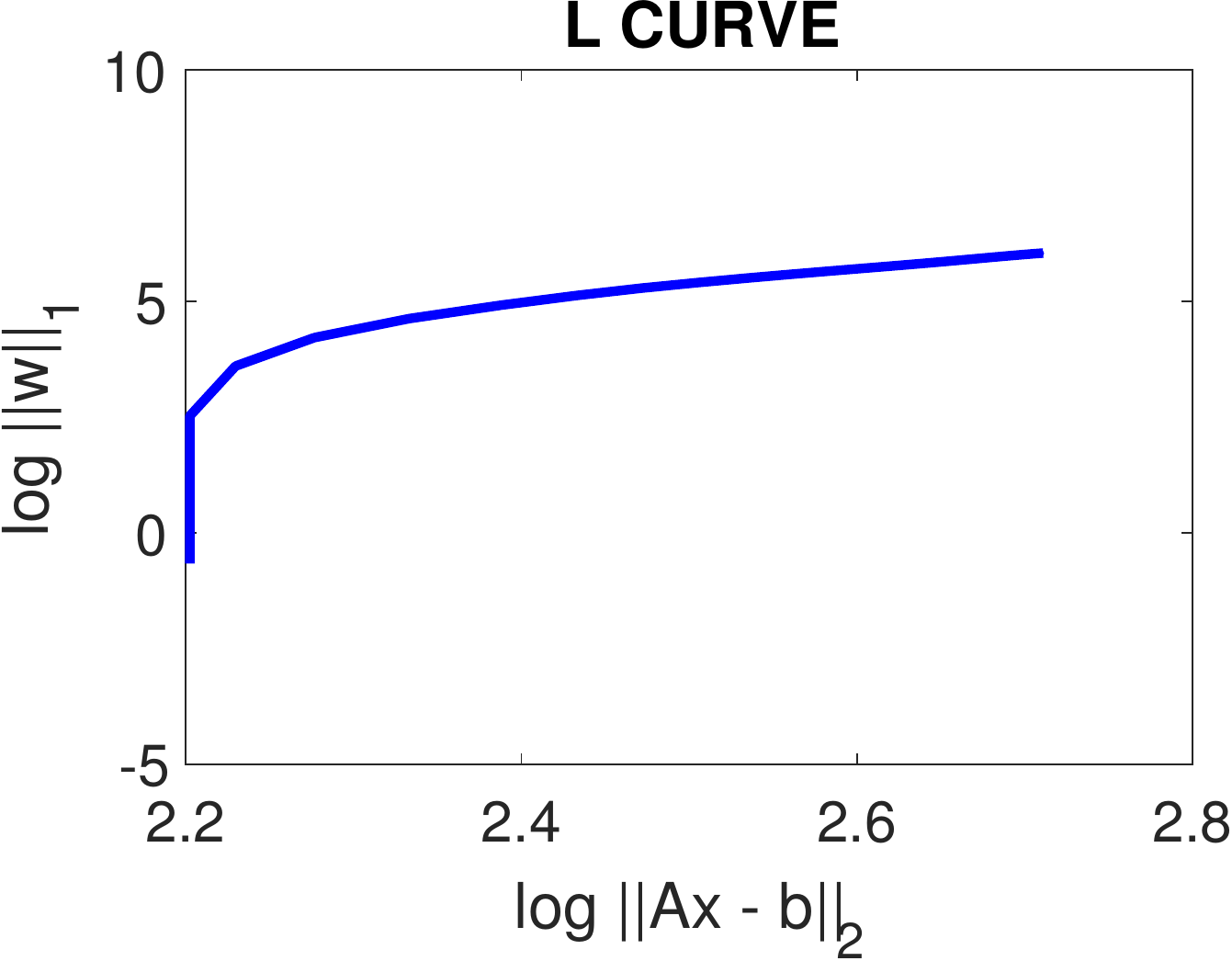}
}
\caption{Comparison of L-curve construction by CONV CG and FISTA schemes. Row 1: L curve illustration 
with a well conditioned matrix. Row 2: CONV CG solution for the worse conditioned matrix. Row 3: FISTA solution for the worse conditoned matrix. \label{fig:test2_lcurve}}
\end{figure*}

\section{Conclusions}

In this paper we present two algorithms based on the CG algorithm, useful in a variety 
of inverse problems. One merit in the methods is in the ability to 
approximately minimize a more general functional, controlled via two parameters 
$l$ and $p$. The functional is useful in a variety of applications, with the $l$ parameter 
controlling the behavior of the residual term (and e.g. the influence of data value 
variations and outliers on the solution) and with $p$ controlling the type of penalty on the components of the solution 
vector (allowing either a minimum norm based penalty or a sparsity promoting penalty term). 
The other merit is in the increased power of the methods per iteration, compared to e.g. thresholding based methods, via the use of 
the heavily researched CG algorithm (and its many possible variants) at each iteration. This allows for the construction of 
approximate regularized solutions as defined by the minimization problem in \eqref{eq:lp_funct}, at fewer iterations.

\bibliographystyle{plain}
\bibliography{ref}

\end{document}